\pdfoutput=1
\RequirePackage[noend]{algpseudocode}
\documentclass[a4paper,english,11pt]{texbase/shinyart} 
\usepackage[utf8]{inputenc}
\usepackage[T1]{fontenc}
\usepackage{csquotes}
\usepackage{babel}
\usepackage{tikz}
\usepackage{float}
\usepackage{enumitem}
\usepackage{comment}
\usepackage{mathtools}
\usepackage{mathbbol}

\usetikzlibrary{external}
\usepackage{pgfplots}
\pgfplotsset{compat=newest}
\pgfplotsset{plot coordinates/math parser=false}
\usetikzlibrary{plotmarks}
\newlength\figureheight
\newlength\figurewidth
\usepackage{grffile}

\floatstyle{ruled}
\newfloat{Algorithm}{tbp!}{loa}
\crefname{Algorithm}{Algorithm}{Algorithms}

\numberwithin{Algorithm}{section}

\usepackage{bbm}
\usepackage{amssymb}
\usepackage{mdframed}



\mdfdefinestyle{examplebg}{
    backgroundcolor=black!7,%
    hidealllines=true,%
    innertopmargin=-.3em,%
    innerbottommargin=.7em,%
    innerleftmargin=.7em,%
    innerrightmargin=.7em,%
}

\surroundwithmdframed[style=examplebg]{example}
\surroundwithmdframed[style=examplebg]{algorithm}

\newcounter{Hequation}

\makeatletter
\g@addto@macro\equation{\stepcounter{Hequation}}
\makeatother

\usepackage[textsize=footnotesize,color=DarkOrange!40]{todonotes}

\newcommand{\term}{\emph}

\newcommand{\field}[1]{\mathbb{#1}}
\newcommand{\N}{\mathbb{N}}

\newcommand{\R}{\field{R}}
\newcommand{\extR}{\overline \R}

\newcommand{\norm}[1]{\|#1\|}

\newcommand{\inv}[1]{#1^{-1}}
\newcommand{\grad}{\nabla}
\newcommand{\freevar}{\,\boldsymbol\cdot\,}

\newcommand{\Union}\bigcup
\newcommand{\Isect}\bigcap
\newcommand{\union}\cup
\newcommand{\isect}\cap
\newcommand{\bigunion}\bigcup
\newcommand{\bigisect}\bigcap

\newcommand{\defeq}{:=}

\newcommand{\subdiff}{\partial}

\DeclareMathOperator{\Dom}{dom}

\makeatletter
\def \uminus@sym{\setbox0=\hbox{$\cup$}\rlap{\hbox 
        to\wd0{\hss\raise0.5ex\hbox{$\scriptscriptstyle{-}$}\hss}}\box0}
    \def \uminus    {\mathrel{\uminus@sym}}
\makeatother

\newcommand{\iprod}[2]{\langle #1,#2\rangle}

\makeatletter
\def \weaktostar@sym{\setbox0=\hbox{$\rightharpoonup$}\rlap{\hbox 
        to\wd0{\hss\raise1ex\hbox{$\scriptscriptstyle{*\,}$}\hss}}\box0}
    \def \weaktostar    {\mathrel{\weaktostar@sym}}
\makeatother

\def\linear{\mathcal{L}}

\newcommand{\setto}{\rightrightarrows}

\def\extR{\overline \R}

\def\opt#1{\widebar #1}
\def\realopt#1{\widehat #1}
\def\this#1{#1^i}
\def\nexxt#1{#1^{i+1}}

\def\optx{{\opt{x}}}

\def\opty{{\opt{y}}}

\def\realoptu{{\realopt{u}}}

\def\realoptx{{\realopt{x}}}

\def\realopty{{\realopt{y}}}

\def\nextu{\nexxt{u}}
\def\nextx{\nexxt{x}}
\def\nextz{\nexxt{z}}
\def\nexty{\nexxt{y}}

\def\thisu{\this{u}}
\def\thisx{\this{x}}
\def\thisz{\this{z}}
\def\thisy{\this{y}}

\def\tauTest{\phi}

\def\sigmaTest{\psi}

\def\gap{\mathcal{G}}

\def\Space{U}

\newcommand{\Test}{Z}
\newcommand{\Precond}{M}
\newcommand{\Step}{W}

\DeclareFontFamily{U}{mathx}{\hyphenchar\font45}
\DeclareFontShape{U}{mathx}{m}{n}{<-> mathx10}{}
\DeclareSymbolFont{mathx}{U}{mathx}{m}{n}
\DeclareMathAccent{\widebar}{0}{mathx}{"73}

\def\prev#1{#1^{i-1}}
\def\prevu{\prev{u}}
\def\prevx{\prev{x}}
\def\prevy{\prev{y}}
\def\thisz{\this{z}}
\def\nextz{\nexxt{z}}

\def\PrecondPlus{\check\Precond}

\DeclareMathOperator{\prox}{prox}

\def\bar{\widebar}

\makeatletter
\renewcommand{\todo}[2][]{\tikzexternaldisable\@todo[#1]{#2}\tikzexternalenable}
\makeatother

\def\thetitle{Inertial, corrected, primal-dual proximal splitting}

\title{\thetitle}

\makeatletter
\AtBeginDocument{
\hypersetup{
    pdftitle = {\thetitle}
    pdfauthor = {T.~Valkonen}
}
}
\makeatother

\newenvironment{demonstration}[1][Demonstration]{\begin{proof}[#1]}{\end{proof}}

\providecommand{\orcid}[1]{\mbox{\scshape\sffamily orcid:}\,\href{https://orcid.org/#1}{\detokenize{#1}}}

\begin{document}

\date{2018-04-24 (revised 2020-02-12)}
\author{
    Tuomo Valkonen\thanks{Department of Mathematics and Statistics, University of Helsinki, Finland \emph{and} ModeMat, Escuela Politécnica Nacional, Quito, Ecuador;
    \emph{previously} Department of Mathematical Sciences, University of Liverpool, United Kingdom. \email{tuomo.valkonen@iki.fi}, \orcid{0000-0001-6683-3572}}
    }

\maketitle

\begin{abstract}
    We study inertial versions of primal-dual proximal splitting, also known as the Chambolle--Pock method. Our starting point is the preconditioned proximal point formulation of this method. By adding correctors corresponding to the anti-symmetric part of the relevant monotone operator, using a FISTA-style gap unrolling argument, we are able to derive gap estimates instead of merely ergodic gap estimates. Moreover, based on adding a diagonal component to this corrector, we are able to combine strong convexity based acceleration with inertial acceleration.
    We test our proposed method on image processing and inverse problems problems, obtaining convergence improvements for sparse Fourier inversion and Positron Emission Tomography.
\end{abstract}

\section{Introduction}
\label{sec:introduction}

For convex, proper, and lower semicontinuous $G: X \to \extR$ and $F^*: Y \to \extR$, and a bounded linear operator $K \in \linear(X; Y)$ on Hilbert spaces $X$ and $Y$, we will derive inertial primal-dual optimisation methods for the problem
\begin{equation}
    \label{eq:prob-basic}
    \min_{x \in X}~ G(x) + F(Kx).
\end{equation}

If $K$ is the identity, and $F$ is smooth, a classical algorithm for the iterative solution of \eqref{eq:prob-basic} is the forward--backward splitting method $\nextx \defeq \prox_{\tau G}(\thisx - \tau \grad F(\thisx))$, where $\tau L < 1$ for $L$ the Lipschitz factor of $\grad F$.
That is, we take proximal steps with respect to $G$, and gradient steps with respect to $F$.
The proximal step needs to be efficiently realisable, i.e., $G$ needs to be ``prox-simple''.
If no strong convexity is present, the iterates of the forward--backward splitting generally converge weakly, and the function values at the rate $O(1/N)$. By applying \term{inertia}, the latter can be improved to $O(1/N^2)$. This in essence consists of rebasing the algorithm at an inertial variable $\this{\bar x}$:
\begin{equation}
    \label{eq:fista}
    \nextx \defeq \prox_{\tau G}(\this{\bar x} - \tau \grad F(\this{\bar x})),
    \quad\text{where}\quad
    \this{\bar x} \defeq (1+\alpha_i)\thisx-\alpha_i\prevx
\end{equation}
for suitable inertial parameters $\{\alpha_i\}_{i \in \N}$. In FISTA \cite{beck2009fista}, which itself is an extension of Nesterov's accelerated gradient descent \cite{nesterov1983method}, one would take
\begin{equation}
    \label{eq:fista-sequence}
    \alpha_{i+1} \defeq \lambda_{i+1}(\inv\lambda_i-1)\quad\text{for}\quad \inv\lambda_{i+1} \defeq \sqrt{\lambda_i^{-2}+1/4}+1/2.
\end{equation}
For this scheme, no convergence rates of the iterates themselves are known, although weak convergence can be obtained with small modifications \cite{chambolle2014convergence}.
Several studies have sought to further optimise the inertial parameters; we refer merely to a few of the most recent works \cite{kim2018fista,attouch2018convergence} and references therein.

If $F$ is non-smooth, but $G$ is smooth, we can apply forward--backward splitting or FISTA with the roles of the two functions exchanged. However, if $K$ is not the identity, $F \circ K$ is rarely prox-simple, so these methods are not practically applicable. Nevertheless, denoting by $F^*$ the Fenchel conjugate of $F$, we can reformulate \eqref{eq:prob-basic} as
\begin{equation}
    \label{eq:saddle}
    \min_{x \in X} \max_{y \in Y}~ G(x) + \iprod{Kx}{y} - F^*(y).
\end{equation}
A popular iterative method for this class of problems is the \term{primal-dual proximal splitting} (PDPS), commonly known as the Chambolle--Pock method \cite{chambolle2010first}.
It takes alternate proximal steps with respect to the primal and dual variables $x$ and $y$:
\begin{equation}
    \label{eq:pdps-basic}
    \left\{
    \begin{aligned}
        \nextx & \defeq \prox_{\tau_i G}(\thisx - \tau_i K^* y^{i}),\\
        \nextx_\omega & \defeq \omega_i (\nextx-\thisx)+\nextx, \\
        \nexty & \defeq \prox_{\sigma_{i+1} F^*}(\thisy + \sigma_{i+1} K \nextx_\omega).
    \end{aligned}
    \right.
\end{equation}
In the basic version the over-relaxation parameter $\omega_i \equiv 1$, and the primal and dual step lengths $\tau_i \equiv \tau_0$, $\sigma_i \equiv \sigma_0$ with $\tau_0\sigma_0 \norm{K}^2<1$. This yields $O(1/N)$ convergence rate for an ergodic gap functional, and weak convergence of the iterates. If $G$ is strongly convex with factor $\gamma>0$, an accelerated version updates $\tau_{i+1} \defeq \omega_i \tau_i$ and $\sigma_{i+1} \defeq \sigma_i/\omega_i$ for $\omega_i \defeq 1/\sqrt{1+\gamma\tau_i}$. This yields $O(1/N^2)$ rates for the ergodic gap as well as $\norm{x^N-\realoptx}^2$.

Several recent works \cite{vu2013splitting,condat2013primaldual,chambolle2014ergodic,chan2014inertial,lorenz2014accelerated,alvarez2001inertial} have applied inertia and closely-related over-relaxation \cite{he2012convergence,eckstein1992douglas} to the basic method \eqref{eq:pdps-basic}. For inertia, writing $\thisu=(\thisx, \thisy)$ and $\this{\bar u}=(\this{\bar x}, \this{\bar y})$, similarly to \eqref{eq:fista}, one rebases the algorithm at $\this{\bar u} \defeq (1+\alpha_i)\thisu-\alpha_i\prevu$ in place of $\thisu$.
In \cite{chambolle2014ergodic}, $O(1/N)$ convergence of an \emph{ergodic} gap functional is shown for this method. No $O(1/N^2)$ results are known to us, or results for a non-ergodic gap or iterates. \emph{In this work, we want to improve upon these convergence rate results, possibly by modifying the algorithm.}

A crucial ingredient for inertia to work in \eqref{eq:fista} is a gap unrolling argument.
To demonstrate this argument, we take for simplicity $F=0$. Then \eqref{eq:fista} implies $\nexxt{q} \defeq -\inv\tau(\nextx-\this{\bar x}) \in \subdiff G(\nextx)$. Defining the auxiliary sequence $
\nexxt{\zeta} \defeq \inv \lambda_{i}\nextx-(\inv\lambda_{i}-1)\thisx$, for all $\realoptx \in X$ one then has\footnote{For the inequality apply Pythagoras' identity to convert the right-hand inner product into norms squared. Then cancel repeated terms and estimate the remaining negative terms. The details in abstract form can also be found in \cref{thm:general-inertia}.}
\begin{equation}
    \label{eq:motivation-base-estimate}
     C_0 \defeq \frac{1}{2\tau}\norm{\zeta^0-\realoptx}^2
     \ge
     -\inv\tau\sum_{i=0}^{N-1}\iprod{\nexxt{\zeta}-\this{\zeta}}{\nexxt{\zeta}-\realoptx}
     =
     \sum_{i=0}^{N-1} \inv\lambda_i \iprod{\nexxt{q}}{\nexxt{\zeta}-\realoptx}.
\end{equation}
If we do not apply inertia, that is $\lambda_i \equiv 1$, we have $\nexxt{\zeta}=\nextx$, so by convexity and Jensen's inequality
\begin{equation}
    \label{eq:jensen-arg}
    C_0 \ge \sum_{i=0}^{N-1} \bigl( G(\nextx) - G(\realoptx) \bigr)
    \ge N\bigl( G(\tilde x^N) - G(\realoptx) \bigr),
    \quad\text{where}\quad
    \tilde x^N \defeq \frac{1}{N} \sum_{i=0}^{N-1} \nextx.
\end{equation}
Due to the variable $\tilde x^N$, this $O(1/N)$ estimate is ergodic.
If, on the other hand, we update $\lambda_i$ as in \eqref{eq:fista-sequence}, we can unroll the ergodicity: Since
\[
    \lambda_i(\nexxt{\zeta}-\realoptx)=\lambda_i(\nextx-\realoptx)+(1-\lambda_i)(\nextx-\thisx),
\]
we estimate from \eqref{eq:motivation-base-estimate} by rearrangements and the definition of the subdifferential that%
\begin{subequations}%
\label{eq:intro-prox-inertia-argumentation}%
\begin{equation}
    \begin{split}
    C_0
    &
    \ge
    \sum_{i=0}^{N-1} \lambda_i^{-2}\Bigl[
        \lambda_i\iprod{\nexxt{q}}{\nextx - \realoptx}
        +(1-\lambda_i)\iprod{\nexxt{q}}{\nextx - \thisx}
    \Bigr]
    \\
    &
    \ge
    \sum_{i=0}^{N-1} \lambda_i^{-2}\Bigl[
        \lambda_i(G(\nextx) - G(\realoptx))
        +(1-\lambda_i)(G(\nextx)-G(\thisx))
    \Bigr]
    \\
    &
    =
    \sum_{i=0}^{N-1} \Bigl[
        \lambda_i^{-2}(G(\nextx) - G(\realoptx))
        -(\lambda_i^{-2}-\inv \lambda_i)(G(\thisx)-G(\realoptx))
    \Bigr].
    \end{split}
\end{equation}
Telescoping and the recurrence  $\lambda_{i}^{-2}=\lambda_{i+1}^{-2}-\inv\lambda_{i+1}$ established from \eqref{eq:fista-sequence} now yield
\begin{equation}
    C_0 + \lambda_0^{-2}(1-\lambda_0)(G(x^0)-G(\realoptx))
    \ge
    \lambda_{N-1}^{-2} (G(x^N)-G(\realoptx)).
\end{equation}%
\end{subequations}%
Since the recurrence also implies that $\lambda_{N}$ is of the order $O(1/N^2)$ \cite{beck2009fista}, this gives the improved convergence rate. Similar arguments can be applied to the forward step component $F$, as we will demonstrate in \cref{sec:unrolling} in a more general setting.

How could such argumentation be applied to the PDPS \eqref{eq:pdps-basic}?
It was discovered in \cite{he2012convergence} that the method can be written as the ``preconditioned proximal point method''
\begin{equation}
    \label{eq:pp0}
    0 \in H(\nextu) + \inv \Step_{i+1}\Precond_{i+1}(\nextu-\thisu)
\end{equation}
in the space $\Space \defeq X \times Y$ with the general notation $u=(x,y)$ for the monotone operator $H: \Space \setto \Space$, the linear preconditioner $\Precond_{i+1} \in \linear(\Space; \Space)$, and the step length operator $\Step_{i+1} \in \linear(\Space; \Space)$ defined as
\begin{align}%
    \label{eq:pdps-h}
    H(u) & \defeq
        \begin{pmatrix}
            \subdiff G(x) + K^* y \\
            \subdiff F^*(y) -K x
        \end{pmatrix},
    &
    \Precond_{i+1} &
    \defeq
    \begin{pmatrix}
        I & -\tau_i K^* \\
        -\sigma_{i+1}\omega_i K & I
    \end{pmatrix},
    \quad\text{and}
    &
    \Step_{i+1} & \defeq
    \begin{pmatrix}
        \tau_i I & 0 \\
        0 & \sigma_{i+1} I
    \end{pmatrix}.
\end{align}%
The over-relaxation parameter $\omega_i$ and the step length parameters $\tau_i$ and $\sigma_{i+1}$ are as after \eqref{eq:pdps-basic}.
Clearly $H(u)=\subdiff \hat G(u)+\Gamma u$ for the convex function $\hat G(u) \defeq G(x)+F^*(y)$ and an anti-symmetric operator $\Gamma$. We can thus apply \eqref{eq:intro-prox-inertia-argumentation} to $\hat G$. However, the anti-symmetric operator $\Gamma$ does not arise as a subdifferential, so similar arguments do not apply to it. Indeed, given some inertial parameters $\lambda_i>0$ and the primal-dual auxiliary sequence $\nexxt{z} \defeq \inv \lambda_{i}\nextu-(\inv\lambda_{i}-1)\thisu$, corresponding to $\nexxt{\zeta}$ above, it does not seem possible to develop a useful estimate out of $\sum_{i=0}^{N-1} \inv\lambda_i \iprod{H(\nextu)}{\nextz-\realoptu}$ alone, unless $\lambda_i \equiv 1$.
\emph{In \cref{sec:inertia}, we are therefore going to \term{correct} the inertial scheme against the anti-symmetry of $\Gamma$.} We do this in the context of general proximal point methods for the solution of the variational inclusion $0 \in H(\realoptu)$.
We demonstrate how to \emph{test} for convergence rates of such methods based on the ideas introduced in \cite{tuomov-proxtest} for non-inertial methods.

We adapt and improve the inertial unrolling argument \eqref{eq:intro-prox-inertia-argumentation} to the setting of this testing theory and corrected inertial methods in \cref{sec:unrolling}. With an eye towards convergence rate proofs, we also develop parameter growth estimates, and briefly demonstrate the theory by application to FISTA.
Based on the general results of \cref{sec:inertia,sec:unrolling}, we then develop our proposed \term{inertial, corrected, primal-dual proximal splitting} (IC-PDPS) in  \cref{sec:inertial-pdps}. Using the corrector, we will also be able to incorporate strong convexity based acceleration into the inertial method.
We finish with conclusions and numerical experience in \cref{sec:numerical}.
Readers wishing to simply implement our proposed method, can find it in an explicit and mostly self-contained form near the end in \cref{alg:pdps-inertial-explicit}. Only the step length rules need to be taken from a choice of theorems given in the algorithm description.


\paragraph{Notation}

We write $\extR \defeq [-\infty, \infty]$ for the extended reals and $\linear(X; Y)$ for the space of bounded linear operators between Hilbert spaces $X$ and $Y$. 
The identity operator in any space is $I$.
For $T,S \in \linear(X; X)$, we write $T \ge S$ when $T-S$ is positive semidefinite.
Also for possibly non-self-adjoint $T$, we introduce the inner product $\iprod{x}{z}_T \defeq \iprod{Tx}{z}$, and, for positive semi-definite $T$, the semi-norm $\norm{x}_T \defeq \sqrt{\iprod{x}{x}_T}$.
For a set $A \subset \R$ and a scalar $c \in \R$, we write $A \ge c$ if every element $t \in A$ satisfies $t \ge c$.
We write $H: X \setto Y$ for $H$ being a set-valued map from $X$ to $Y$.

\section{General inertial methods, correctors}
\label{sec:inertia}

We will now study the application of inertia to general preconditioned proximal point schemes, of which the PDPS is an instance.
As we discussed in the Introduction, \cite{he2012convergence} showed that the PDPS \eqref{eq:pdps-basic} can be written as solving $0 \in H(\nextu) + \inv \Step_{i+1}\Precond_{i+1}(\nextu-\thisu)$ for $\nextu$ with the choices \eqref{eq:pdps-h}.
Further developments in \cite{tuomov-proxtest,tuomov-cpaccel} rewrote the method with $\tilde H_{i+1} \defeq \Step_{i+1} H$ as an instance of the more general scheme
\begin{equation}
    \label{eq:ppext}
    \tag{PP}
    0 \in \tilde H_{i+1}(\nextu) + \Precond_{i+1}(\nextu-\thisu)
\end{equation}
that can also model forward steps.
This formulation, with the step length operator $\Step_{i+1}$ uninverted and moved at the front of $H$ turned out to be beneficial for the development of compact convergence rate proofs of the PDPS \cite{tuomov-proxtest}, as well as stochastic extensions that would have $\Step_{i+1}$ non-invertible \cite{tuomov-blockcp}.

In this section, we will study the application of inertia to \eqref{eq:ppext}.
We start by formulating a simple extension of \eqref{eq:fista} to \eqref{eq:ppext}.
As in \eqref{eq:fista-sequence}, writing the inertial parameter as $\alpha_{i+1}=\lambda_{i+1}(\inv\lambda_i-1)$, we now take an invertible linear operator $\Lambda_{i+1}$ as our fundamental inertial parameter. Given an initial iterate $u^0=\bar u^0 \in \Space$, we then rebase $\thisu$ in \eqref{eq:ppext} to $\this{\bar u}$ to obtain the method
\begin{equation}
    \label{eq:general-inertial}
    \left\{
    \begin{aligned}
    0 &\in \tilde H_{i+1}(\nextu) + \Precond_{i+1}(\nextu-\this{\bar u}),
    \\
    \nexxt{\bar u} & \defeq \nextu+\Lambda_{i+2}(\inv\Lambda_{i+1}-I)(\nextu-\thisu).
    \end{aligned}
    \right.
\end{equation}
We assume that $\tilde H_{i+1}: \Space \setto \Space$ and $\Precond_{i+1},\Lambda_{i+1} \in \linear(\Space; \Space)$ on a Hilbert space $\Space$.

\begin{remark}
    The operator $\Lambda_{i+1}$ has the index $i+1$ off-by-one compared to $\lambda_i$ in \eqref{eq:fista} and \eqref{eq:fista-sequence}. This is for consistency with the historical development of the PDPS \eqref{eq:pdps-basic} into the form \eqref{eq:pp0} or \eqref{eq:ppext}: compare \eqref{eq:pdps-h}, where primal step lengths within the step length operator $\Step_{i+1}$ have index $i$, and dual step lengths index $i+1$.
    This will generally be the case: operator indices agree with dual parameter indices, while primal parameter indices will be one less.
    We have not reindexed the parameters to maintain the property $\sigma_i\tau_i=\sigma_0\tau_0$ of the PDPS.
\end{remark}

We want to correct for any anti-symmetric or otherwise challenging components $\Gamma_{i+1} \in \linear(\Space; \Space)$ of $\tilde H_{i+1}$. We therefore introduce the \term{corrector}
\begin{equation}
    \label{eq:inertia-precondplus}
    \PrecondPlus_{i+1} \defeq \Gamma_{i+1}(\inv\Lambda_{i+1}-I),
\end{equation}
and modify \eqref{eq:general-inertial} into the \term{general corrected inertial method}
\begin{equation}
    \label{eq:pp-inertia}
    \tag{PP-I}
    \left\{
    \begin{aligned}
    0 &\in \tilde H_{i+1}(\nextu) + \Precond_{i+1}(\nextu-\this{\bar u})
    +\PrecondPlus_{i+1}(\nextu-\thisu),
    \\
    \nexxt{\bar u} & \defeq \nextu+\Lambda_{i+2}(\inv\Lambda_{i+1}-I)(\nextu-\thisu).
    \end{aligned}
    \right.
\end{equation}

In this section, our task is to develop general convergence estimates for \eqref{eq:pp-inertia}, which we will then use to prove convergence rates of more specific instances of the general method, in particular the inertial, corrected, PDPS in \cref{sec:inertial-pdps}.
To interpret the main assumption of our abstract convergence estimate, and for later use, we recall the following three-point inequality:

\begin{lemma}
    \label{lemma:smoothness-three-point}
    Let $F: X \to \extR$ be proper, convex, lower semicontinuous with $\grad F$ $L$-Lipschitz. Then
    \begin{equation*}
        \iprod{\grad F(z)}{x-\realoptx}
        \ge
        F(x) - F(\realoptx) -  \frac{L}{2}\norm{x-z}^2
        \quad (\realoptx, z, x \in X).
    \end{equation*}
\end{lemma}

\begin{proof}
    Since $F$ has $L$-Lipschitz gradient, it is smooth in the sense of convex analysis (also known as satisfying the \term{descent inequality}),
    $F(z)-F(x) \ge \iprod{\grad F(z)}{z-x} - \frac{L}{2}\norm{x-z}^2$.
    By convexity $F(\realoptx)-F(z) \ge \iprod{\grad F(z)}{\realoptx-z}$.
    Summing these two estimates, we obtain the claim.
\end{proof}

We will develop our convergence estimates following the testing framework of \cite{tuomov-proxtest}. The idea introduced there was to pick a suitably designed testing operator $\Test_{i+1} \in \linear(\Space; \Space)$, and then apply the testing functional $u \mapsto \iprod{u}{\nextu-\realoptu}_{\Test_{i+1}}$ to both sides of \eqref{eq:ppext}. An almost trivial argument based on a simple assumption on $\tilde H_{i+1}$ and Pythagoras' (three-point) identity would then show that $\Test_{i+1}\Precond_{i+1}$ forms a local metric that measures convergence rates.
However, presently, we cannot in general obtain estimates on the principal sequence $\{\thisu\}_{i \in \N}$. Rather, we will obtain estimates on the auxiliary sequence $\{\thisz\}_{i \in \N}$, defined through
\begin{equation}
    \label{eq:inertia-z}
    z^0 \defeq u^0
    \quad\text{and}\quad
    \nextz \defeq \inv\Lambda_{i+1}\nextu-(\inv\Lambda_{i+1}-I)\thisu
    \quad (i \in \N).
\end{equation}
This adds some additional complexity to the main condition \eqref{eq:inertia-h-cond} of the next theorem. We will motivate the condition after the proof.

\begin{theorem}
    \label{thm:general-inertia}
    On a Hilbert space $\Space$, for $i=0,\ldots,N-1$, let $\tilde H_{i+1}: \Space \setto \Space$ as well as $\Precond_{i+1}, \Test_{i+1}, \Gamma_{i+1}, \Lambda_i, \in \linear(\Space; \Space)$ with $\Lambda_i$ invertible.
    Given an initial iterate $u^0=\bar u^0 \in \Space$, let $\{\nextu\}_{i \in \N}$
    be defined through the solution of \eqref{eq:pp-inertia}, and the auxiliary sequence $\{\thisz\}_{i \in \N}$ by \eqref{eq:inertia-z}.
    Suppose, for $i=1,\ldots,N-1$, that $\Test_{i+1}\Precond_{i+1} \ge 0$ is self-adjoint, and for some  $\realoptu \in \Space$ and a placeholder real value $\mathcal{V}_{i+1}(\realoptu) \in \R$ that
    \begin{multline}
        \label{eq:inertia-h-cond}
        \iprod{\tilde H_{i+1}(\nextu)-\Gamma_{i+1}(\nextu-\realoptu)}{\nextz-\realoptu}_{\Lambda_{i+1}^*\Test_{i+1}}
        \ge
        \mathcal{V}_{i+1}(\realoptu)
        - \frac{1}{2}\norm{\nextz-\thisz}_{\Lambda_{i+1}^*\Test_{i+1}\Precond_{i+1}\Lambda_{i+1}}^2
    \end{multline}
    and
    \begin{equation}
        \label{eq:inertia-metric-cond}
        \Lambda_{i+1}^*\Test_{i+1}(\Precond_{i+1}\Lambda_{i+1}+2\Gamma_{i+1})
        \ge \Lambda_{i+2}^*\Test_{i+2}\Precond_{i+2}\Lambda_{i+2}.
    \end{equation}
    Then
    \begin{equation}
        \label{eq:convergence-result-inertia-h}
        \frac{1}{2}\norm{z^N-\realoptu}^2_{\Lambda_{N+1}^*\Test_{N+1}\Precond_{N+1}\Lambda_{N+1}}
        +
        \sum_{i=0}^{N-1} \mathcal{V}_{i+1}(\realoptu)
        \le
        \frac{1}{2}\norm{z^0-\realoptu}^2_{\Lambda_{1}^*\Test_{1}\Precond_{1}\Lambda_{1}}
        \quad
        (N \ge 1).
    \end{equation}
\end{theorem}

\begin{proof}
    Application of $\iprod{\freevar}{\Test_{i+1}^*\Lambda_{i+1}(\nextz-\realoptu)}$ to the main inclusion of \eqref{eq:pp-inertia} yields for some $\nexxt{q} \in \tilde H_{i+1}(\nextu)$ that
    \begin{equation}
        \label{eq:ppinertia-tested}
        0 = \iprod{\nexxt{q}+\Precond_{i+1}(\nextu-\this{\bar u})+\PrecondPlus_{i+1}(\nextu-\thisu)}{\Lambda_{i+1}(\nextz-\realoptu)}_{\Test_{i+1}}.
    \end{equation}
    By \eqref{eq:inertia-z} and \eqref{eq:inertia-precondplus}, we deduce that
    \begin{equation*}
        \PrecondPlus_{i+1}(\nextu-\thisu)
        =\Gamma_{i+1}(\inv\Lambda_{i+1}-I)(\nextu-\thisu)
        =\Gamma_{i+1}(\nextz-\nextu).
    \end{equation*}
    By the definition $\this{\bar u}$ in \eqref{eq:pp-inertia}, and of the auxiliary sequence $\{\nextz\}_{i \in \N}$ in \eqref{eq:inertia-z}, taking $u^{-1} \defeq u^0$, moreover,
    \begin{equation}
        \label{eq:z-baru-relationship}
        \begin{split}
        \Lambda_{i+1}(\nextz-\thisz)
        &
        =\nextu-(I-\Lambda_{i+1})\thisu-\Lambda_{i+1}[\inv\Lambda_i\thisu-(\inv\Lambda_i-I)\prevu]
        \\
        &
        =\nextu-[I-\Lambda_{i+1}+\Lambda_{i+1}\inv\Lambda_i]\thisu-\Lambda_{i+1}(I-\inv\Lambda_i)\prevu
        \\
        &
        = \nextu - \this{\bar u}.
        \end{split}
    \end{equation}
    Therefore, we transform \eqref{eq:ppinertia-tested} into
    \begin{equation}
        \label{eq:ppinertia-tested-2}
        0 = \iprod{\nexxt{q}+\Precond_{i+1}\Lambda_{i+1}(\nextz-\thisz)+\Gamma_{i+1}(\nextz-\nextu)}{\Lambda_{i+1}(\nextz-\realoptu)}_{\Test_{i+1}}.
    \end{equation}
    Writing for brevity $A \defeq \Lambda_{i+1}^*\Test_{i+1}\Precond_{i+1}\Lambda_{i+1}$, which by assumption is self-adjoint and positive semi-definite, the standard three-point formula or Pythagoras' identity states
    \begin{equation*}
       \iprod{\nextz-\thisz}{\nextz-\realoptu}_{A}
       = \frac{1}{2}\norm{\nextz-\thisz}_{A}^2
           - \frac{1}{2}\norm{\thisz-\realoptu}_{A}^2
           + \frac{1}{2}\norm{\nextz-\realoptu}_{A}^2.
    \end{equation*}
    We therefore transform \eqref{eq:ppinertia-tested-2} into
    \begin{equation*}
        \begin{split}
        0 & = \iprod{\nexxt{q}+\Gamma_{i+1}(\nextz-\nextu)}{\nextz-\realoptu}_{\Lambda_{i+1}^*\Test_{i+1}}
            + \frac{1}{2}\norm{\nextz-\thisz}_{\Lambda_{i+1}^*\Test_{i+1}\Precond_{i+1}\Lambda_{i+1}}^2
            \\ & \phantom{ = }
            - \frac{1}{2}\norm{\thisz-\realoptu}_{\Lambda_{i+1}^*\Test_{i+1}\Precond_{i+1}\Lambda_{i+1}}^2
            + \frac{1}{2}\norm{\nextz-\realoptu}_{\Lambda_{i+1}^*\Test_{i+1}\Precond_{i+1}\Lambda_{i+1}}^2.
        \end{split}
    \end{equation*}
    Using \eqref{eq:inertia-h-cond}, we obtain
    \begin{equation*}
        \begin{split}
        0 & \ge \mathcal{V}_{i+1}(\realoptu)
            - \frac{1}{2}\norm{\thisz-\realoptu}_{\Lambda_{i+1}^*\Test_{i+1}\Precond_{i+1}\Lambda_{i+1}}^2
            \\ \MoveEqLeft[-1]
            + \frac{1}{2}\norm{\nextz-\realoptu}_{\Lambda_{i+1}^*\Test_{i+1}\Precond_{i+1}\Lambda_{i+1}}^2
            + \iprod{\Gamma_{i+1}(\nextz-\realoptu)}{\nextz-\realoptu}_{\Lambda_{i+1}^*\Test_{i+1}}.
        \end{split}
    \end{equation*}
    Using \eqref{eq:inertia-metric-cond} and summing over $i=0,\ldots,N-1$ establishes \eqref{eq:convergence-result-inertia-h}.
\end{proof}

\begin{remark}
    Consider $\Lambda_{i+1}=\Test_{i+1}=\Precond_{i+1}=I$. Then \eqref{eq:inertia-h-cond} reads
    \begin{equation}
        \label{eq:inertia-h-cond-simplified}
        \iprod{\tilde H_{i+1}(\nextu)}{\nextu-\realoptu}
        \ge
        \mathcal{V}_{i+1}(\realoptu)
        + \iprod{\nextu-\realoptu}{\nextu-\realoptu}_{\Gamma_{i+1}}
        - \frac{1}{2}\norm{\nextu-\thisu}^2.
    \end{equation}
    With $\tau L \le 1$ and $\Gamma_{i+1}=0$, take first $\tilde H_{i+1}(u)=\tau \grad F(\thisu)$, fixing $\grad F$ to be evaluated at the previous iteration to obtain a gradient descent method. Then it is easy to see how this estimate with $\mathcal{V}_{i+1}(\realoptu)=\tau[F(\nextu)-F(\realoptu)]$ follows from \cref{lemma:smoothness-three-point}. Thus $\mathcal{V}_{i+1}$ measures function value differences. Similarly, if $\tilde H_{i+1}(u)=\tau \subdiff G(u)$ for non-smooth but ($\gamma$-strongly) convex $G$, we can take $\Gamma_{i+1}=\tau\gamma I$ in \eqref{eq:inertia-h-cond-simplified}. We can also combine  $\tilde H_{i+1}(u)=\tau[\subdiff G(u) + \grad F(\thisu)]$ to obtain forward--backward splitting.
    In other words, \eqref{eq:inertia-h-cond} is an operator-relative inertia-aware convexity and smoothness condition, where the variable $\mathcal{V}_{i+1}(\realoptu)$ can be used to model function value and other gap estimates. We will need this operator-relativity to apply distinct inertial and testing parameters on the primal and dual variables of the PDPS; compare the block structure of \eqref{eq:pdps-h}.

    Minding this interpretation of \eqref{eq:inertia-h-cond}, the claim \eqref{eq:convergence-result-inertia-h} of the theorem with additional positivity and growth assumptions can thus be used to show the convergence of the sum of the value estimates $\mathcal{V}_{i+1}(\realoptu)$ to zero, as well as $z^N \to \realoptu$.
    Our task in the following is to obtain that growth, and to unroll the sum into a simple estimate.
\end{remark}

\section{Unrolling and parameter growth estimates}
\label{sec:unrolling}

To develop the inertial, corrected, primal-dual proximal splitting method, we will seek to satisfy the conditions of \cref{thm:general-inertia} for an algorithm inspired by the proximal point interpretation of the PDPS. Before we do this, in this section, we will prove general inertial unrolling arguments (\cref{sec:unrolling0}), refining \eqref{eq:intro-prox-inertia-argumentation} to the testing framework. We also prove parameter growth estimates with an eye towards converge rate proofs (\cref{sec:growth}), and demonstrate how our corrector term allows combining inertia with strong convexity-based acceleration (\ref{sec:inertia-sc}).
We finish by applying these estimates to the FISTA to demonstrate how it fits into our overall approach (\cref{sec:inertia-sc}). This also demonstrates how our approach works without the additional challenges of the primal-dual setup.

\subsection{Scalar parameter choices}
\label{sec:scalar-choices}

To place the general estimates that make up the major part of this section into context, we start by specialising \cref{thm:general-inertia} to $\Lambda_{i+1}=\lambda_i I$, $\Test_{i+1}=\tauTest_i I$, $\Step_{i+1}=\tau_i I$, $\Precond_{i+1} = I$, and  $\Gamma_{i+1} \defeq \gamma\tau_i I$ for some scalars $\lambda_i,\tauTest_i,\tau_i>0$ and $\gamma \ge 0$. Also taking $\tilde H_{i+1}(x) \defeq \tau_i(\subdiff G(x)+\grad F(\thisx))$, we immediately rewrite \eqref{eq:pp-inertia}, with change of symbol\footnotemark{} $u$ into $x$, as
\begin{equation}
    \label{eq:pp-inertia-scalar}
    \tag{PP-i}
    \left\{
    \begin{aligned}
    0 & \in \tau_i[\subdiff G(\nextx) + \grad F(\thisx)] + (\nextx-\this{\bar x})
    +\gamma\tau_i(\inv\lambda_i-1)(\nextx-\thisx),
    \\
    \nexxt{\bar x} & \defeq \nextx+\lambda_{i+1}(\inv\lambda_i-1)(\nextx-\thisx).
    \end{aligned}
    \right.
\end{equation}%
\addtocounter{footnote}{-1}%
Moreover, with the change of symbol\footnotemark{} of $z$ into $\zeta$, the auxiliary sequence defined in \eqref{eq:inertia-z} becomes
\begin{equation}
    \label{eq:inertia-z-scalar}
    \zeta^0 \defeq x^0
    \quad\text{and}\quad
    \nexxt{\zeta} \defeq \inv\lambda_i\nextx-(\inv\lambda_i-1)\thisx
    \quad (i \in \N).
\end{equation}
Immediately, \cref{thm:general-inertia} specialises into:

\footnotetext{We reserve the symbols $u$ and $z$ for the abstract (\cref{sec:inertia}) and primal-dual (\cref{sec:inertial-pdps}) problems. In the latter we take primal-dual pairs $u=(x,y)$ and $z=(\zeta,\eta)$, so the primal variables match the symbols of this section.}

\begin{corollary}
    \label{cor:simplealg}
    On a Hilbert space $X$, let $G: X \to \extR$ and $F: X \to \R$ be convex, proper, and lower semicontinuous, with $F$ differentiable.
    Let $\lambda_i,\tauTest_i,\tau_i>0$, ($i=0,\ldots,N-1$), and $\gamma \ge 0$.
    For an initial iterate $x^0=\bar x^0 \in X$, let $\{\nextx\}_{i=0}^{N-1}$ be generated by \eqref{eq:pp-inertia-scalar}, and the auxiliary sequence $\{\this{\zeta}\}_{i=0}^{N-1}$ by \eqref{eq:inertia-z-scalar}.
    For each $i=0,\ldots,N-1$, suppose for some  $\realoptx \in X$ and a placeholder value $\mathcal{V}_{i+1}(\realoptx) \in \R$, we have the estimate
    \begin{equation}
        \label{eq:cor:simplealg:condition}
        \tauTest_i\lambda_i\tau_i \iprod{\subdiff G(\nextx)+\grad F(\thisx)- \gamma(\nextx-\realoptx)}{\nexxt{\zeta}-\realoptx}
        \ge \mathcal{V}_{i+1}(\realoptx) - \frac{\lambda_i^2\tauTest_i}{2}\norm{\nexxt{\zeta}-\this{\zeta}}^2
    \end{equation}
    and the inequality
    \begin{equation}
        \label{eq:simplealg:accelrule}
        \lambda_{i+1}^2\tauTest_{i+1} \le \lambda_i^2\tauTest_i(1+2\gamma\inv\lambda_i\tau_i).
    \end{equation}
    Then
    \begin{equation}
        \label{eq:cor:simplealg:claim}
        \frac{\tauTest_N\lambda_N^2}{2}\norm{\zeta^N-\realoptx}^2
        + \sum_{i=0}^{N-1} \mathcal{V}_{i+1}(\realoptx)
        \le
        \frac{\tauTest_0\lambda_0^2}{2}\norm{\zeta^0-\realoptx}^2
        \quad (N \ge 1).
    \end{equation}
\end{corollary}


\subsection{Inertial unrolling}
\label{sec:unrolling0}

We start by refining the proximal step inertial unrolling argument \eqref{eq:intro-prox-inertia-argumentation}. We will in \cref{sec:growth} see that the recurrence inequality \eqref{eq:h-sum-function-lambda-recurrence} assumed by the next lemma generalises the recurrence $\lambda_{i}^{-2}=\lambda_{i+1}^{-2}-\inv\lambda_{i+1}$ from the Introduction, satisfied by the FISTA.

\begin{lemma}
    \label{lemma:h-sum-function}
    Let $G: X \to \extR$ be convex, proper, and lower semicontinuous.
    Suppose $\lambda_i \in (0, 1]$ and $\tauTest_i,\tau­_i > 0$ satisfy the recurrence inequality
    \begin{equation}
        \label{eq:h-sum-function-lambda-recurrence}
        \tauTest_{i+1}\tau_{i+1}(1-\lambda_{i+1}) \le \tauTest_i\tau_i
        \quad (i=0,\ldots,N-1).
    \end{equation}
    For any given $\{\thisx\}_{i=0}^N$, let the auxiliary variables $\{\this{\zeta}\}_{i=0}^N$ be generated by \eqref{eq:inertia-z-scalar}.
    Assume $\subdiff G(\nextx)$ to be non-empty for $i=0,\ldots,N-1$, and $\realoptx \in \inv{[\subdiff G]}(0)$.
    Then
    \begin{equation}
        \label{eq:h-sum-function}
        \begin{split}
        \sum_{i=0}^{N-1}
            \inf_{\nexxt{q} \in \subdiff G(\nextx)}
            \tauTest_i\tau_i\lambda_i \iprod{\nexxt{q}}{\nexxt{\zeta} - \realoptx}
        &
        \ge
        \tauTest_{N-1}\tau_{N-1}(G(x^N) - G(\realoptx))
        \\ \MoveEqLeft[-1]
        -\tauTest_0\tau_0(1-\lambda_0)(G(x^0)-G(\realoptx)).
        \end{split}
    \end{equation}
\end{lemma}

\begin{proof}
    For all $i=0,\ldots,N-1$, pick $\nexxt{q} \in \subdiff G(\nextx)$, and define
    \[
        s_N^G
        =
        \sum_{i=0}^{N-1} \tauTest_i\tau_i\lambda_i\iprod{\nexxt{q}}{\nexxt{\zeta}-\realoptx}.
    \]
    Then we need to show that
    \begin{equation}
        \label{eq:h-sum-function-claim1}
        s_N^G
        \ge
        \tauTest_{N-1}\tau_{N-1}(G(x^N)-G(\realoptx))
        -\tauTest_0\tau_0(1-\lambda_0)(G(x^0)-G(\realoptx)).
    \end{equation}
    Observe that the auxiliary variables $\{\nexxt{\zeta}\}_{i=0}^{N-1}$ satisfy
    \begin{equation}
        \label{eq:inertia-nextz-realoptu-useful-for-unrolling-scalar}
        \lambda_{i}(\nexxt{\zeta}-\realoptx)
        =\lambda_{i}(\nextx-\realoptx)+(1-\lambda_{i})(\nextx-\thisx).
    \end{equation}
    With this and the convexity of $G$, we estimate
    \begin{equation}
        \label{eq:h-sum-function-first-ineq}
        \begin{split}
        s_N^G
        &
        =
        \sum_{i=0}^{N-1} \tauTest_i\tau_i\Bigl[
            \lambda_i\iprod{\nexxt{q}}{\nextx - \realoptx}
            +(1-\lambda_i)\iprod{\nexxt{q}}{\nextx - \thisx}
        \Bigr]
        \\
        &
        \ge
        \sum_{i=0}^{N-1} \tauTest_i\tau_i\Bigl[
            \lambda_i(G(\nextx) - G(\realoptx))
            +(1-\lambda_i)(G(\nextx)-G(\thisx))
        \Bigr]
        \\
        &
        =
        \sum_{i=0}^{N-1} \Bigl[
            \tauTest_i\tau_i(G(\nextx) - G(\realoptx))
            -\tauTest_i\tau_i(1-\lambda_i)(G(\thisx)-G(\realoptx))
        \Bigr].
        \end{split}
    \end{equation}
    Since $G(\thisx)\ge G(\realoptx)$, the recurrence inequality \eqref{eq:h-sum-function-lambda-recurrence} together with a telescoping argument now give \eqref{eq:h-sum-function-claim1}.
\end{proof}

We can also include a forward step in the unrolling argument:

\begin{lemma}
    \label{lemma:h-sum-function-smooth}
    Let $G, F: X \to \extR$ be convex, proper, and lower semicontinuous. Suppose $F$ has $L$-Lipschitz gradient, and that $\lambda_i \in (0, 1]$ and $\tauTest_i,\tau_i >0$ satisfy the recurrence inequality \eqref{eq:h-sum-function-lambda-recurrence} for $i=0,\ldots,N-1$.
    For any given $\{\thisx\}_{i=0}^N$, let the auxiliary variables $\{\this{\zeta}\}_{i=0}^N$ be generated by \eqref{eq:inertia-z-scalar}.
    Assume $\subdiff G(\nextx)$ to be non-empty for all $i=0,\ldots,N-1$, and that $\realoptx \in \inv{[\subdiff G+\grad F]}(0)$.
    Then
    \begin{multline}
        \label{eq:h-sum-function-smooth}
        \sum_{i=0}^{N-1} \inf_{\nexxt{q} \in \subdiff G(\nextx)} \left[
            \tauTest_i\tau_i\lambda_i \iprod{\nexxt{q}+ \grad F(\this{\bar x})}{\nexxt{\zeta} - \realoptx}
            +\frac{\tauTest_i\tau_i \lambda_i^2 L}{2}\norm{\nexxt{\zeta}-\this{\zeta}}^2
            \right]
        \\
        \ge
        \tauTest_{N-1}\tau_{N-1}[(G+F)(x^N) - (G+F)(\realoptx)]
        -\tauTest_{0}\tau_{0}(1-\lambda_0)[(G+F)(x^0) - (G+F)(\realoptx)].
    \end{multline}
\end{lemma}

\begin{proof}
    Similarly to \eqref{eq:z-baru-relationship},
    $
        \lambda_i(\nexxt{\zeta}-\this{\zeta})=(\nextx-\this{\bar x}).
    $
    By \eqref{eq:inertia-nextz-realoptu-useful-for-unrolling-scalar} and \cref{lemma:smoothness-three-point}, therefore
    \[
        \begin{split}
        s_N^F
        & \defeq \sum_{i=0}^{N-1}\left[
            \tauTest_i\tau_i\lambda_i \iprod{\grad F(\this{\bar x})}{\nexxt{\zeta} - \realoptx}
            +\frac{\tauTest_i\tau_i \lambda_i^2 L}{2}\norm{\nexxt{\zeta}-\this{\zeta}}^2
            \right]
        \\
        &
        =
        \sum_{i=0}^{N-1} \tauTest_i\tau_i\left[
            \lambda_i\iprod{\grad F(\this{\bar x})}{\nextx - \realoptx}
            +(1-\lambda_i)\iprod{\grad F(\this{\bar x})}{\nextx - \thisx}
            + \frac{L}{2}\norm{\nextx-\this{\bar x}}^2
        \right]
        \\
        &
        \ge
        \sum_{i=0}^{N-1} \tauTest_i\tau_i\left[
            \lambda_i(F(\nextx) - F(\realoptx))
            +(1-\lambda_i)(F(\nextx)-F(\thisx))
        \right]
        \\
        &
        =
        \sum_{i=0}^{N-1} \left[
            \tauTest_i\tau_i(F(\nextx) - F(\realoptx))
            -\tauTest_i\tau_i(1-\lambda_i)(F(\thisx)-F(\realoptx))
        \right].
        \end{split}
    \]
    Picking $\nexxt{q} \in \subdiff G(\nextx)$, ($i=0,\ldots,N-1$), and summing with the estimate \eqref{eq:h-sum-function-first-ineq} for $G$, we deduce
    \[
        s_N^G + s_N^F
        \ge
        \sum_{i=0}^{N-1} \tauTest_i\tau_i\left[
            [(G+F)(\nextx) - (G+F)(\realoptx)]
            -(1-\lambda_i)[(G+F)(\thisx)-(G+F)(\realoptx)]
        \right].
    \]
    Since $(G+F)(\thisx) \ge (G+F)(\realoptx)$, the recurrence inequality \eqref{eq:h-sum-function-lambda-recurrence} together with a telescoping argument now give the claim.
\end{proof}



\subsection{Parameter growth estimates}
\label{sec:growth}

As suggested by the unrolled estimates \eqref{eq:h-sum-function} and \eqref{eq:h-sum-function-smooth}, we want to make $\tauTest_{N-1}\tau_{N-1}$ grow as fast as possible while satisfying \eqref{eq:h-sum-function-lambda-recurrence} and $\lambda_i \in (0, 1]$. We now develop such estimates through a series of lemmas.
The first of these lemmas with $\epsilon=0$ is the FISTA rate argument \cite[Lemma 4.3]{beck2009fista}.

\begin{lemma}
    \label{lemma:lambda-recurrence}
    With $\lambda_0=1$, suppose $\lambda_{i}^{-2}-\epsilon\inv\lambda_i=\lambda_{i+1}^{-2}-\inv\lambda_{i+1}$ for some $\epsilon \in [-1, 1]$ and all $i=0,\ldots,N-1$.
    Then $\{\inv\lambda_i\}_{i \in \N}$ is non-decreasing, $\inv \lambda_N \ge 1+ (1-\epsilon)N/2$, and we equivalently define $\lambda_{i+1}$ through
    \begin{equation}
        \label{eq:lambda-recurrence}
        \lambda_{i+1}=\frac{2}{1+\sqrt{1+4(\lambda_i^{-2}-\epsilon\inv\lambda_i)}}.
    \end{equation}
\end{lemma}

\begin{proof}
    The update \eqref{eq:lambda-recurrence} is a simple solution of the quadratic equation $\lambda_{i}^{-2}-\epsilon\inv\lambda_i=\lambda_{i+1}^{-2}-\inv\lambda_{i+1}$. 
    The latter also rearranges as
    \begin{equation}
        \label{eq:lambad-recurrence-v2}
        \lambda_{i+1}^{-2}-\inv\lambda_{i+1}=\lambda_i^{-2}-\inv\lambda_i+(1-\epsilon)\inv\lambda_i.
    \end{equation}
    This shows that $\{t_i \defeq \lambda_i^{-2}-\inv\lambda_i\}_{i \in \N}$ is non-decreasing. 
    Since $\inv\lambda_i=\tfrac{1}{2}(1+\sqrt{1+4t_i})$, we obtain the claim that $\{\inv\lambda_i\}_{i \in \N}$ is non-decreasing.

    We still need to prove the rate-of-growth claim. Since $\lambda_0=1$, \eqref{eq:lambad-recurrence-v2} also yields
    \[
        \lambda_N^{-2}-\inv\lambda_N=\sum_{i=0}^{N-1}(1-\epsilon)\inv\lambda_i.
    \]
    Let us make the inductive assumption that $\inv\lambda_i \ge 1+(1-\epsilon)i/2$ for $i=0,\ldots,N-1$. Clearly this holds for $N=0$ by the choice $\lambda_0=1$, taking care of the inductive base. For the inductive step, we get from above and $\epsilon \in [-1, 1]$ that
    \[
        \lambda_N^{-2}-\inv\lambda_N
        \ge (1-\epsilon)N + \frac{(1-\epsilon)^2}{4}N(N-1)
        \ge \frac{1-\epsilon}{2}N+\left(\frac{1-\epsilon}{2}\right)^2 N^2.
    \]
    This quadratic inequality together with $\lambda_N>0$ imply
    \[
        \inv\lambda_N
        \ge
        \frac{1+\sqrt{1+2(1-\epsilon)N+(1-\epsilon)^2N^2}}{2}
        =1+\frac{1-\epsilon}{2}N,
    \]
    which verifies the inductive step and establishes the claim.
\end{proof}

\begin{lemma}
    \label{lemma:tau-lambda-gamma-satisfaction}
    The sequence $\{\tauTest_i\tau_i\}_{i \in \N}$ is non-decreasing and the conditions \eqref{eq:simplealg:accelrule} and \eqref{eq:h-sum-function-lambda-recurrence} hold, more precisely
    \begin{equation}
        \label{eq:tau-lambda-gamma-satisfaction-eq}
        \lambda_{i+1}^2\tauTest_{i+1}=\lambda_i^2\tauTest_i(1+2\gamma\inv\lambda_i\tau_i)
        \quad\text{and}\quad
        \tauTest_{i+1}\tau_{i+1}(1-\lambda_{i+1})=(1-\epsilon\lambda_i)\tauTest_i\tau_i
    \end{equation}
    for all $i \in \N$ for some $\epsilon \in [0, 1)$ in the following cases:
    \begin{enumerate}[label=(\roman*)]
        \item\label{item:tau-lambda-gamma-satisfaction-1}
        If $\gamma=0$ and we take $\tau_i \equiv \tau$ for any $\tau>0$; $\tauTest_i \defeq \lambda_i^{-2}$; $\tauTest_0=\lambda_0=1$, and update $\lambda_{i+1}$ for any $\epsilon \in [0, 1]$ according to \eqref{eq:lambda-recurrence}.
        Then also
        \[
            \tauTest_N\tau_N \ge (1-\epsilon)^2N^2\tau/4
            \quad\text{and}\quad
            \lambda_N^2\tauTest_N = 1
            \quad (N \in \N).
        \]

        \item\label{item:tau-lambda-gamma-satisfaction-2}
        If $\gamma>0$ and we take $\lambda_i \equiv \lambda \in (0, 1)$ and $\tau_i \equiv \tau \defeq \lambda^2/[2\gamma(1-\lambda)]$ constants, and $\tauTest_{i+1}=c\tauTest_i$ with $c \defeq (1-\epsilon\lambda)/(1-\lambda) > 1$ for any $\epsilon \in [0, 1)$ and $\tauTest_0>0$. Then also
        \[
            \tauTest_N\tau_N \ge \tauTest_0\tau c^N
            \quad\text{and}\quad
            \lambda_N^2\tauTest_N \ge \lambda^2\tauTest_0 c^N
            \quad (N \in \N).
        \]

        \item\label{item:tau-lambda-gamma-satisfaction-3}
        If we are constrained to have $\tauTest_i=c_0 \tau_i^{-2}$ for some constant $c_0>0$, and with $\lambda_0=1$, $\tau_0>0$ and $\epsilon \in [0, 1)$ update
        \begin{equation}
            \label{eq:tau-lambda-gamma-satisfaction-3-rules}
            \tau_{i+1}=\frac{1-\lambda_{i+1}}{1-\epsilon\lambda_i}\tau_i
            \quad\text{and}\quad
            \lambda_{i+1}=\frac{\sqrt{\lambda_i^2+2\gamma\lambda_i\tau_i}}{1-\epsilon\lambda_i+\sqrt{\lambda_i^2+2\gamma\lambda_i\tau_i}}.
        \end{equation}
        Then, for some constants $c, c'>0$, for all $N \in \N$, also
        \begin{align*}
            \tauTest_N\tau_N & \ge c' N^2
            &\text{and}&&
            \lambda_N^2\tauTest_N & \ge cN^2
            &&(\gamma>0),
            \\
            \tauTest_N\tau_N & \ge (1-\epsilon)\inv\tau_0 N
            &\text{and}&&
            \lambda_N^2\tauTest_N & = c_0 \tau_0^{-2}
            &&(\gamma=0).
        \end{align*}
    \end{enumerate}
\end{lemma}

The choice $\epsilon=0$ in \cref{item:tau-lambda-gamma-satisfaction-3} would be the simplest, and also optimal in the sense that both \eqref{eq:simplealg:accelrule} and \eqref{eq:h-sum-function-lambda-recurrence} would hold as equalities.
However, we will see that a non-zero choice performs significantly better in practise---with the same asymptotic guarantees.

\begin{proof}
    It is clear that the inequalities \eqref{eq:simplealg:accelrule} and \eqref{eq:h-sum-function-lambda-recurrence} follow from \eqref{eq:tau-lambda-gamma-satisfaction-eq} and $\epsilon \ge 0$.

    \cref{item:tau-lambda-gamma-satisfaction-1}
    Since $\gamma=0$, the first part of \eqref{eq:tau-lambda-gamma-satisfaction-eq} holds when $\tauTest_i\lambda_i^2=\tauTest_0\lambda_0^2$. This follows from our choices $\tauTest_0=\lambda_0=1$ and $\tauTest_i=\lambda_i^{-2}$. Inserting $\tauTest_i=\lambda_i^{-2}$ and $\tau_i \equiv \tau$, the second part of \eqref{eq:tau-lambda-gamma-satisfaction-eq} reduces to $\lambda_{i}^{-2}-\epsilon\inv\lambda_i=\lambda_{i+1}^{-2}-\inv\lambda_{i+1}$.
    \Cref{lemma:lambda-recurrence} shows that $\tauTest_{N}\tau_{N}=\lambda_{N}^{-2}\tau \ge \tau (1-\epsilon)^2N^2/4$. This is the claimed estimate.
    Moreover, since $\{\inv\lambda_i\}_{i \in \N}$ is non-decreasing by \cref{lemma:lambda-recurrence}, we see that $\{\tauTest_i\tau_i\}_{i \in \N}$ is non-decreasing.

    \cref{item:tau-lambda-gamma-satisfaction-2}
    The second part of \eqref{eq:tau-lambda-gamma-satisfaction-eq} agrees with the chosen update rule for $\tauTest_{i+1}$.
    Inserting this rule into the first part of \eqref{eq:tau-lambda-gamma-satisfaction-eq} and using the fact that also $\tau_i \equiv \tau$, we see the latter to be satisfied if $(1-\lambda)(1+2\gamma\inv\lambda\tau)=1$. This is satisfied by our chosen
    $\tau=\lambda^2/[2\gamma(1-\lambda)]$.
    Since $\tau$ and $\lambda$ are constants, the claimed growth estimates follow from $\tauTest_N=\tauTest_0 c^N$.
    Clearly $\{\tauTest_i\tau_i=\tauTest_0\tau c^i\}_{i \in \N}$ is non-decreasing.

    \cref{item:tau-lambda-gamma-satisfaction-3}
    Finally, with $\tauTest_i= c_0 \tau_i^{-2}$, \eqref{eq:tau-lambda-gamma-satisfaction-eq} holds if
    \begin{equation}
        \label{eq:inertia-accel-constr-conds}
        \lambda_{i+1}^2\tau_{i+1}^{-2} = \lambda_i^2\tau_i^{-2}(1+2\gamma\inv\lambda_i\tau_i)
        \quad\text{and}\quad
        \tau_{i+1}^{-1}(1-\lambda_{i+1}) = (1-\epsilon\lambda_i) \tau_i^{-1}.
    \end{equation}
    The latter agrees with our update rule for $\tau_{i+1}$. Inserting this, the former holds if $\lambda_{i+1}(1-\epsilon\lambda_i)=(1-\lambda_{i+1})\sqrt{\lambda_i^2+2\gamma\lambda_i\tau_i}$. This is satisfied by our update rule for $\lambda_{i+1}$.

    To derive the growth estimates, suppose first that $\gamma>0$. With $\theta_i \defeq \lambda_i\tau_i^{-1}=c_0^{-1/2}\lambda_i\tauTest_i^{1/2}$, the first part of \eqref{eq:inertia-accel-constr-conds} reads $\theta_{i+1}^2=\theta_i^2(1+2\gamma\theta_i)$.
    Thus $\{\theta_i\}_{i \in \N}$ is non-decreasing, moreover, the recurrence is of the same form as the standard acceleration rule for the PDPS, where we would have $\tauTest_i=c_0 \tau_i^{-2}$ in place of $\theta_i$; compare \cref{sec:introduction} and \cite{chambolle2010first,tuomov-cpaccel}.
    Hence $\theta_i \ge c i$ for some constant $c>0$.
    Since $\lambda_i^2\tauTest_i=c_0 \theta_i^2$, this gives one of the claimed rates.
    From the second part of \eqref{eq:inertia-accel-constr-conds},
    \begin{equation}
        \label{eq:tau-lambda-gamma-satisfaction-3-1}
        \tau_{i+1}^{-1}
        =(1-\epsilon\lambda_i)\tau_i^{-1}+\lambda_{i+1}\tau_{i+1}^{-1}
        =\tau_i^{-1}-\epsilon\theta_i+\theta_{i+1}.
    \end{equation}
    Repeating this recursively, since $\theta_0=\inv\tau_0$, for some constant $c'>0$,
    \begin{equation}
        \label{eq:tau-lambda-gamma-satisfaction-3-2}
        \inv \tau_N
        = \inv\tau_0 - \epsilon\theta_0 + \theta_N + \sum_{i=1}^{N-1}(1-\epsilon)\theta_i
        \ge (1-\epsilon)\inv\tau_0 + c N + \sum_{i=0}^{N-1} (1-­\epsilon) c i \ge c' N^2.
    \end{equation}
    Therefore also $\tauTest_N\tau_N=c_0\inv\tau_N$ has the claimed growth estimate.

    Suppose then that $\gamma=0$. We obtain \eqref{eq:tau-lambda-gamma-satisfaction-3-1} as above, however now with constant $\theta_i \equiv \lambda_i\inv\tau_i=\lambda_0\inv\tau_0=\inv\tau_0$. Therefore, arguing as in \eqref{eq:tau-lambda-gamma-satisfaction-3-2} only yields $\inv \tau_N \ge (1-\epsilon)\inv\tau_0N$. This is again the claimed growth estimate.

    Finally, since  $\{\theta_i\}_{i \in \N}$ is in both cases ($\gamma=0$ or $\gamma>0$) non-decreasing, we see from \eqref{eq:tau-lambda-gamma-satisfaction-3-1} that $\{\inv\tau_i\}_{i \in \N}$ and consequently $\{\tauTest_i\tau_i=c_0\inv\tau_i\}_{i \in \N}$ are non-decreasing.
\end{proof}

\subsection{Combining inertia with strong convexity}
\label{sec:inertia-sc}

Let $G: X \to \extR$ be proper, lower semicontinuous, and (strongly) convex with parameter $\gamma \ge 0$. We now demonstrate with a simple inertial proximal point method, \eqref{eq:pp-inertia-scalar} with $F=0$, how we are able to incorporate strong convexity based acceleration with inertia. We do this by considering the convex functions
\begin{equation}
    \label{eq:g-gamma}
    G_\gamma(x; \realoptx) \defeq G(x) - \frac{\gamma}{2}\norm{x-\realoptx}^2.
\end{equation}
Indeed, $0 \in \subdiff G(\realoptx)$ if and only if $0 \in \subdiff G_\gamma(\realoptx; \realoptx)$ with
$
    \subdiff G_\gamma(x; \realoptx)=\subdiff G(x)-\gamma(x-\realoptx).
$
\Cref{lemma:h-sum-function} applied to $G_\gamma(\freevar; \realoptx)$ thus shows
\begin{equation}
    \label{eq:h-sum-function-gamma}
    \sum_{i=0}^{N-1} \mathcal{V}_{i+1}(\realoptx)
    \ge
    \tauTest_{N-1}\tau_{N-1}[G_\gamma(x^N;\realoptx) - G_\gamma(\realoptx;\realoptx)]
    -\tauTest_0\tau_0(1-\lambda_0)[G_\gamma(x^0; \realoptx)-G_\gamma(\realoptx; \realoptx)]
\end{equation}
for
\[
    \mathcal{V}_{i+1}(\realoptx) \defeq \inf_{\nexxt{q} \in \subdiff G(\nextx)} \tauTest_i\tau_i\lambda_i \iprod{\nexxt{q}-\gamma(\nextx-\realoptx)}{\nexxt{\zeta} - \realoptx}.
\]
This choice of $\mathcal{V}_{i+1}(\realoptx)$ by definition verifies \eqref{eq:cor:simplealg:condition}.
It therefore follows from \cref{cor:simplealg} with $F \equiv 0$ that
\begin{multline*}
    \frac{\tauTest_N\lambda_N^2}{2}\norm{\zeta^N-\realoptx}^2
    + \tauTest_{N-1}\tau_{N-1}[G_\gamma(x^N;\realoptx) - G_\gamma(\realoptx;\realoptx)]
    \\
    \le
    \frac{\tauTest_0\lambda_0^2}{2}\norm{\zeta^0-\realoptx}^2
    +\tauTest_0\tau_0(1-\lambda_0)[G_\gamma(x^0; \realoptx)-G_\gamma(\realoptx; \realoptx)]
    \quad (N \ge 1).
\end{multline*}
If we choose our parameters according to \cref{lemma:tau-lambda-gamma-satisfaction}\,\ref{item:tau-lambda-gamma-satisfaction-2} , then $\tauTest_N\tau_N$ and $\lambda_N^2\tauTest_N$ grow exponentially. Crucially $G_\gamma(x^N;\realoptx) \ge G_\gamma(\realoptx;\realoptx)$, so this implies linear convergence of $G_\gamma(x^N;\realoptx) \to G_\gamma(\realoptx;\realoptx)$ and of the auxiliary sequence $\zeta^N \to \realoptx$.

\subsection{Connections}

We now discuss how our results relate to known algorithms.

\begin{example}[FISTA]
    \label{example:fista}
    Let $G: X \to \extR$ and $F: X \to \R$ be convex, proper, and lower semicontinuous with $\grad F$ existing and $L$-Lipschitz.
    If $\gamma=0$, take $\tau_i \equiv \tau \in (0, 1/L]$, and $\lambda_{i+1}$ by \eqref{eq:lambda-recurrence} for $\lambda_0=1$ and $\epsilon=0$.
    Then given initial iterates $\bar x^0 = x^0$, \eqref{eq:pp-inertia-scalar} becomes the inertial forward--backward splitting or FISTA
    \[
        \left\{
        \begin{aligned}
        \nextx & \defeq \prox_{\tau G}(\this{\bar x}-\tau \grad F(\this{\bar x})),
        \\
        \nexxt{\bar x} & \defeq \nextx+\lambda_{i+1}(\inv\lambda_i-1)(\nextx-\thisx).
        \end{aligned}
        \right.
    \]
    We have $G(x^N) + F(x^N) \to G(\realoptx) + F(\realoptx)$ at the rate $O(1/N^2)$ for any minimiser $\realoptx$ of $G+F$.
\end{example}

\begin{demonstration}
    The algorithm is clear from \cref{eq:pp-inertia-scalar}.
    \Cref{lemma:tau-lambda-gamma-satisfaction}\cref{item:tau-lambda-gamma-satisfaction-1} shows that \eqref{eq:simplealg:accelrule} and \eqref{eq:h-sum-function-lambda-recurrence} are satisfied, $\tauTest_N\tau_N \ge \tau N^2/4$, and $\lambda_N^2\tauTest_N = 1$.
    Clearly \eqref{eq:cor:simplealg:condition} holds with
    \[
        \mathcal{V}_{i+1}(\realoptx) \defeq \inf_{\nexxt{q} \in \subdiff G(\nextx)} \tauTest_i\tau_i\lambda_i \iprod{\nexxt{q}+\grad F(\this{\bar x})-\gamma(\nextx-\realoptx)}{\nexxt{\zeta} - \realoptx}
        + \frac{\lambda_i^2\tauTest_i}{2}\norm{\nexxt{\zeta}-\this{\zeta}}^2.
    \]
    Using in \eqref{eq:h-sum-function-smooth} the fact that $\tau_i L \le 1$ , similarly to \cref{sec:inertia-sc}, we may therefore refer to \cref{lemma:h-sum-function-smooth,cor:simplealg} to verify the estimate
    \begin{multline}
        \label{eq:cor:simplealg:claim-fista}
        \frac{\tauTest_N\lambda_N^2}{2}\norm{\zeta^N-\realoptx}^2
        +\tauTest_{N-1}\tau_{N-1}[(\hat G_\gamma+F)(x^N) - (\hat G_\gamma+F)(\realoptx)]
        \\
        \le
        \frac{\tauTest_0\lambda_0^2}{2}\norm{\zeta^0-\realoptx}^2
        +\tauTest_{0}\tau_{0}(1-\lambda_0)[(\hat G_\gamma+F)(x^0) - (\hat G_\gamma+F)(\realoptx)],
    \end{multline}
    where we use the short-hand notation $\hat G_\gamma \defeq G_\gamma(\freevar; \realoptx)$.
    Inserting our choice $\lambda_0=1$ with $\gamma=0$, $\zeta^0=x^0$ and the growth estimates from above, we obtain
    \[
        (G+F)(x^N)
        \le
        (G+F)(\realoptx)
        + \frac{2}{(N-1)^2\tau}\norm{x^0-\realoptx}^2
        \quad (N \ge 2).
    \]
    This verifies the claimed convergence rates.
\end{demonstration}

\begin{example}[FISTA combined with strong convexity]
    In \cref{example:fista}, suppose in addition that $G$ is strongly convex with parameter $\gamma > 0$. Take $0 < \lambda \le
    \sqrt{L^{-2}\gamma^2+2\inv L\gamma} - \inv L\gamma$
    and $\tau \defeq \lambda^2/[2\gamma(1-\lambda)]$. Also let $\tilde\tau \defeq \tau/[1+(\inv\lambda-1)\gamma\tau]$.
    Then  \eqref{eq:pp-inertia-scalar} with $\lambda_i \equiv \lambda$ and $\tau_i \equiv \tau$ becomes
    \[
        \left\{
        \begin{aligned}
        \nextx & \defeq \prox_{\tilde\tau G}(\this{\tilde x}-\tilde\tau\grad F(\this{\bar x})),
        \\
        \nexxt{\bar x} & \defeq \nextx+\lambda(\inv\lambda-1)(\nextx-\thisx),
        \\
        \nexxt{\tilde x} & \defeq (\nexxt{\bar x}+\gamma\tau(\inv\lambda-1)\nextx)/(1+\gamma\tau(\inv\lambda-1)).
        \end{aligned}
        \right.
    \]
    Both $G_\gamma(x^N; \realoptx) + F(x^N) \to G_\gamma(\realoptx; \realoptx) + F(\realoptx)$ and $\zeta^N \to \realoptx$ at the linear rate $O((1-\lambda)^N)$.
\end{example}

\begin{demonstration}
   To derive the claimed algorithm, we divide the first step of \eqref{eq:pp-inertia-scalar} by $1+\gamma\tau(\inv\lambda-1)$. This yields $\this{\tilde x} - \tilde\tau \grad F(\this{\bar x}) \in \tilde\tau \subdiff G(\nextx)+\nextx$; the rest follows from the definition of $\this{\bar x}$ and the proximal map.
   Observe that $\lambda \in (0, 1)$.  \Cref{lemma:tau-lambda-gamma-satisfaction}\cref{item:tau-lambda-gamma-satisfaction-2} with $\epsilon=0$ and $\tauTest_0=1$ thus proves \eqref{eq:simplealg:accelrule} and \eqref{eq:h-sum-function-lambda-recurrence}, and shows that $\tauTest_N\tau_N \ge \tau/(1-\lambda)^N$ and $\lambda_N^2\tauTest_N \ge \lambda^2/(1-\lambda)^N$. As in \cref{example:fista}, we now obtain \eqref{eq:cor:simplealg:claim-fista} provided $\tau L \le 1$. With our chosen $\tau = \lambda^2/[2\gamma(1-\lambda)]$, this constraints resolves as our assumed upper bound $\lambda \le \sqrt{L^{-2}\gamma^2+2\inv L\gamma} - \inv L\gamma$.
   With the growth rates above and $\tau = \lambda^2/[2\gamma(1-\lambda)]$, \eqref{eq:cor:simplealg:claim-fista} after some rearrangements yields
   \begin{multline*}
        (1-\lambda)\norm{\zeta^N-\realoptx}^2
        +\inv\gamma[(\hat G_\gamma+F)(x^N) - (\hat G_\gamma+F)(\realoptx)]
        \\
        \le
        (1-\lambda)^{N+1}\left(
        \norm{x^0-\realoptx}^2
        +\inv\gamma[(\hat G_\gamma+F)(x^0) - (\hat G_\gamma+F)(\realoptx)]
        \right).
   \end{multline*}
   Thus the claimed convergence rates hold.
\end{demonstration}

\begin{remark}
    Douglas--Rachford splitting for the problem $\min_{x \in X}~F(x)+G(x)$ reads
    \begin{equation*}
        \left\{
        \begin{aligned}
        \nextx &= \prox_{\gamma F}(\this{v}),\\
        \nexty &= \prox_{\gamma G}(2\nextx - \this{v}),\\
        \nexxt{v} &= \this{v} + \nexty - \nextx.
        \end{aligned}
        \right.
    \end{equation*}
    It can be presented in the form \eqref{eq:ppext} with $\tilde H_{i+1}=H$ for $u=(x, y, v)$ and
    \begin{equation*}
        H(u) \defeq
        \begin{pmatrix}
            \tau \subdiff F(x) + y - v \\
            \tau \subdiff G(y) + v - x \\
            x - y
        \end{pmatrix}
        \quad\text{and}\quad
        \Precond_{i+1} \defeq
        \begin{pmatrix}
            0 & 0 & 0 \\
            0 & 0 & 0 \\
            0 & 0 & I
        \end{pmatrix}.
    \end{equation*}
    With
    \begin{equation*}
        \Gamma_{i+1}
        \defeq
        \begin{pmatrix}
            0 & I & -I \\
            -I & 0 & I \\
            I & -I & 0
        \end{pmatrix},
        \quad\text{taking instead}\quad
        \Precond_{i+1} \defeq
        \begin{pmatrix}
            0 & 0 & 0 \\
            0 & 0 & 0 \\
            0 & 0 & \inv\lambda_{i+1} I
        \end{pmatrix}.
    \end{equation*}
    our approach can be used to construct a corrected inertial Douglas--Rachford splitting. We will, however, not pursue this. Instead, in the next section we take the primal-dual proximal splitting as an example of an algorithm with a non-trivial corrector and $\Gamma_{i+1}$.

    Inertial Douglas--Rachford splitting has previously been studied in  \cite{patrinos2014douglas}. The ``corrected'' algorithm derived from our approach will be different. Another accelerated approach is considered in \cite{bredies2014preconditioned}. They apply Douglas--Rachford splitting to $H$ defined in \eqref{eq:pdps-h} by writing it in the form $H(u)=\subdiff \hat G(u)+\Gamma u$ for the convex function $\hat G(u) \defeq G(x)+F^*(y)$ and an anti-symmetric operator $\Gamma$, as we did in \cref{sec:introduction}. What this ingenious approach yields is essentially a doubly over-relaxed PDPS.
\end{remark}

\section{Inertial primal-dual proximal splitting}
\label{sec:inertial-pdps}

We now return to the saddle point problem \eqref{eq:saddle}.
We suppose $G: X \to \extR$ and $F^*: Y \to \extR$ are (strongly) convex with factors $\gamma, \rho \ge 0$, and $K \in \linear(X; Y)$. Recalling \eqref{eq:pdps-h}, for some step length, testing, and inertial parameters $\tau_i,\sigma_{i+1},\tauTest_i,\sigmaTest_{i+1},\lambda_i,\mu_{i+1}>0$, we then take $\tilde H_{i+1}=\inv \Step_{i+1} H$ as well as
\begin{subequations}
\label{eq:pdps-inertia-structure}
\begin{align}%
    H(u) & \defeq
        \begin{pmatrix}
            \subdiff G(x) + K^* y \\
            \subdiff F^*(y) -K x
        \end{pmatrix},
    &
    \Step_{i+1} & \defeq \begin{pmatrix} \tau_i I & 0 \\ 0 & \sigma_{i+1} I \end{pmatrix},
    &
    \Test_{i+1} & \defeq \begin{pmatrix} \tauTest_i I & 0 \\ 0 & \sigmaTest_{i+1} I \end{pmatrix},
    \\
    \Gamma_{i+1} & \defeq \begin{pmatrix} \gamma\tau_i I & \tau_i K^* \\ -\sigma_{i+1} K & \rho\sigma_{i+1} I \end{pmatrix},
    &\Lambda_{i+1} & \defeq \begin{pmatrix} \lambda_i I & 0 \\ 0 & \mu_{i+1} I \end{pmatrix},&
    \text{and}
    \\
    \label{eq:pdps-inertia-precond}
    \Precond_{i+1} & = \begin{pmatrix} I & -\inv\mu_{i+1}\tau_i K^* \\
    -\inv\lambda_i\sigma_{i+1}\omega_i K & I \end{pmatrix}
    \span\quad\text{for}\span
    &\omega_i & \defeq \frac{\lambda_i\tauTest_i\tau_i}{\lambda_{i+1}\tauTest_{i+1}\tau_{i+1}}.
\end{align}
\end{subequations}
We then observe from \eqref{eq:inertia-precondplus} that the corrector
\begin{equation}
    \label{eq:pdps-inertia-precondplus}
    \PrecondPlus_{i+1}=\Gamma_{i+1}(\inv\Lambda_{i+1}-I)
    =
    \begin{pmatrix}
        \gamma\tau_i (\inv\lambda_i-1)I & \tau_i (\inv\mu_{i+1}-1)K^* \\
        -\sigma_{i+1} (\inv\lambda_i-1)K
        & \rho\sigma_{i+1}(\inv\mu_{i+1}-1) I
    \end{pmatrix}.
\end{equation}

We need to satisfy the conditions of \cref{thm:general-inertia} for this setup, in particular \eqref{eq:inertia-h-cond} for some $\mathcal{V}_{k+1}(\realoptu)$ and $\realoptu \in \inv H(0)$, and show that the estimate \eqref{eq:convergence-result-inertia-h} is useful, in particular that $\Test_{i+1}\Precond_{i+1}$ and $\sum_{i=0}^{N-1} \mathcal{V}_{i+1}(\realoptu)$ are positive, the former grows at a good rate, and that the latter becomes a useful gap functional.
In first instance, we intend to develop it into (a multiple of) the Lagrangian (duality) gap
\begin{equation}
    \label{eq:gap}
    \gap(x, y; \realoptx, \realopty) \defeq
    \bigl(G(x) + \iprod{\realopty}{K x}  - F^*(\realopty)\bigr)
    -\bigl(G(\realoptx) + \iprod{y}{K \realoptx} -  F^*(y)\bigr).
\end{equation}
Recalling $G_\gamma$ and $(F^*)_\rho$ defined in \eqref{eq:g-gamma}, we also introduce the strong convexity adjusted gap
\begin{equation}
    \label{eq:gap-sc}
    \gap_{\gamma,\rho}(x, y; \realoptx, \realopty) \defeq
    \bigl(G_\gamma(x; \realoptx) + \iprod{\realopty}{K x}  - (F^*)_\rho(\realopty; \realopty)\bigr)
    -\bigl(G_\gamma(\realoptx; \realoptx) + \iprod{y}{K \realoptx} -  (F^*)_\rho(y; \realopty)\bigr).
\end{equation}
Since the problem $\min_x \max_y~ G_\gamma(x; \realoptx)+\iprod{Kx}{y} - (F^*)_\rho(y; \realopty)$ has the solution $(\realoptx, \realopty)$, it is clear that $\gap_{\gamma,\rho}$ is non-negative, and zero at $(\realoptx, \realopty)$.
Before proving convergence we derive an explicit algorithm from \eqref{eq:pdps-inertia-structure}.

\subsection{Algorithm derivation}
\label{sec:pdps-derivation}

With the structure \eqref{eq:pdps-inertia-structure} fixed, we are ready to develop the skeleton of an explicit algorithm out of \eqref{eq:pp-inertia}.
Since \eqref{eq:pp-inertia} updates
\[
    (\this{\bar x},\this{\bar y}) = \this{\bar u} \defeq \thisu+ \Lambda_{i+1}(\inv\Lambda_i-I)(\thisu-\prevu),
\]
we have
\begin{align}
    \label{eq:pdps-inertial-variables}
    \this{\bar x} & = \thisx+\lambda_i(\inv\lambda_{i-1}-1)(\thisx-\prevx),
    \quad\text{and}
    &
    \this{\bar y} & = \thisy+\mu_{i+1}(\inv\mu_i-1)(\thisy-\prevy).
\end{align}
Using \eqref{eq:pdps-inertia-precondplus} and \eqref{eq:pdps-inertia-structure} we expand \eqref{eq:pp-inertia} as
\[
    \left\{
    \begin{aligned}
    0 & \in \tau_i \subdiff G(\nextx) + \tau_i K^*\nexty
        +(\nextx-\this{\bar x})
        -\inv\mu_{i+1}\tau_i K^*(\nexty-\this{\bar y})
        \\ \MoveEqLeft[-1]
        +\gamma\tau_i(\inv\lambda_i-1)(\nextx-\thisx)
        +\tau_i(\inv\mu_{i+1}-1)K^*(\nexty-\thisy),
        \\
    0 & \in \sigma_{i+1} \subdiff F^*(\nexty) - \sigma_{i+1} K\nextx
        -\inv\lambda_i\sigma_{i+1}\omega_i K(\nextx-\this{\bar x})
        + (\nexty-\this{\bar y})
        \\ \MoveEqLeft[-1]
        +\rho\sigma_{i+1}(\inv\mu_{i+1}-1)(\nexty-\thisy)
        -\sigma_{i+1}(\inv\lambda_i-1)K(\nextx-\thisx).
    \end{aligned}
    \right.
\]
The second line in both inclusions comes from the corrector term.
Collecting all instances of the same iterate together, this can be simplified as
\begin{equation}
    \label{eq:pdps-derivation-step-2}
    \left\{
    \begin{aligned}
    0 & \in \tau_i \subdiff G(\nextx)
        +[1+\gamma\tau_i(\inv\lambda_i-1)]\nextx
        -[\this{\bar x}+\gamma\tau_i(1-\lambda_i)\thisx]
        \\ \MoveEqLeft[-1]
        +\tau_i K^*[\inv\mu_{i+1}\this{\bar y}-(\inv\mu_{i+1}-1)\thisy],
        \\
    0 & \in \sigma_{i+1} \subdiff F^*(\nexty)
        +[1+\rho\sigma_{i+1}(\inv\mu_{i+1}-1)]\nexty
        -[\this{\bar y}+\rho\sigma_{i+1}(\inv\mu_{i+1}-1)\thisy]
        \\ \MoveEqLeft[-1]
        -\inv\lambda_i\sigma_{i+1}(1+\omega_i) K\nextx
        +\inv\lambda_i\sigma_{i+1}\omega_i K \this{\bar x}
        +\sigma_{i+1}(\inv\lambda_i-1)K\thisx.
    \end{aligned}
    \right.
\end{equation}
Using \eqref{eq:pdps-inertial-variables} we can write,
\[
    \begin{split}
    \this{\tilde y}
    & \defeq \inv\mu_{i+1}\this{\bar y}-(\inv\mu_{i+1}-1)\thisy
    = \inv\mu_{i+1}\thisy+(\inv\mu_i-1)(\thisy-\prevy)-(\inv\mu_{i+1}-1)\thisy
    \\
    &
    = \thisy + (\inv\mu_i-1)(\thisy-\prevy).
    \end{split}
\]
Similarly, defining
\[
    \nexxt{\tilde x} \defeq \nextx+(\inv\lambda_i-1)(\nextx-\thisx),
\]
we can write
\[
    \begin{split}
    \nextx_{\omega} & \defeq
        \inv\lambda_i(1+\omega_i)\nextx
        -\inv\lambda_i\omega_i \this{\bar x}
        -(\inv\lambda_i-1)\thisx
        \\
        & =
        [\inv\lambda_i\nextx-(\inv\lambda_i-1)\thisx]
        +\omega_i[\inv\lambda_i\nextx-\inv\lambda_i\this{\bar x}]
        \\
        & =
        \nexxt{\tilde x}
        +\omega_i[\inv\lambda_i\nextx-(\inv\lambda_i-1)\thisx+(\inv\lambda_i-1)\thisx-\inv\lambda_i\this{\bar x}]
        \\
        &
        =\nexxt{\tilde x} + \omega_i(\nexxt{\tilde x}-\this{\tilde x}).
    \end{split}
\]
Also introducing
\[
    \tilde\tau_i \defeq \tau_i/[1+\gamma\tau_i(\inv\lambda_i-1)]
    \quad\text{and}\quad
    \tilde\sigma_{i+1} \defeq \sigma_{i+1}/[1+\rho\sigma_{i+1}(\inv\mu_{i+1}-1)],
\]
we now rewrite \eqref{eq:pdps-derivation-step-2} as
\[
    \left\{
    \begin{aligned}
    0 & \in \tau_i \subdiff G(\nextx)
        +(\tau_i/\tilde\tau_i)\nextx
        -[\this{\bar x}+\gamma\tau_i(\inv\lambda_i-1)\thisx]
        +\tau_i K^* \this{\tilde y},
        \\
    0 & \in \sigma_{i+1} \subdiff F^*(\nexty)
        +(\sigma_{i+1}/\tilde\sigma_{i+1})\nexty
        -[\this{\bar y}+\rho\sigma_{i+1}(\inv\mu_{i+1}-1)\thisy]
        -\sigma_{i+1} K \nextx_\omega.
    \end{aligned}
    \right.
\]
Multiplying, respectively, by $\tilde\tau_i/\tau_i$ and $\tilde\sigma_{i+1}/\sigma_{i+1}$, and recalling \eqref{eq:pdps-inertial-variables}, we obtain the proximal updates of \cref{alg:pdps-inertial-explicit}, which we have written somewhat more compactly by additionally introducing the iterates $\this x_\gamma$ and $\this y_\rho$. The updates of $\nexxt{\bar x}$, $\nexxt{\tilde x}$, $\nexxt{\bar y}$, and $\nexxt{\tilde y}$ in the main step of the method are simply the definitions from above.
The step length parameters will still need to be determined from one of the theorems referenced in \cref{alg:pdps-inertial-explicit}.
Observe how the ``corrected'' inertial variables $\nexxt{\tilde x}$ and $\nexxt{\tilde y}$ differ from the standard inertial variables $\nexxt{\bar x}$ and $\nexxt{\bar y}$.

Before developing specific rules for the step lengths and inertial parameters, we still need to provide the estimate \eqref{eq:inertia-h-cond}. This process will produce additional conditions on the parameters.

\begin{Algorithm}
    \caption{Inertial, corrected, primal-dual proximal splitting (IC-PDPS)}
    \label{alg:pdps-inertial-explicit}
    \begin{algorithmic}[1]
        \Require
        On Hilbert spaces $X$ and $Y$, a linear operator $K \in \linear(X; Y)$ and convex, proper, and lower semicontinuous $G: X \to \extR$ and $F^*: Y \to \extR$ with factors $\gamma,\rho \ge 0$ of (strong) convexity.
        \State
        Determine step length and inertial parameters $\{(\tau_i,\sigma_{i+1}, \lambda_i, \mu_{i+1}, \omega_i)\}_{i \in \N}$ from a suitable one among \cref{thm:step-length-rules-gamma=0-rho=0,thm:step-length-rules-gamma>0-rho=0,thm:step-length-rules-gamma=0-rho>0,thm:step-length-rules-gamma>0-rho>0}.
        \State Pick initial iterates $\tilde x^0 \defeq \bar x^0 \defeq x^0 \in X$, and $\tilde y^0 \defeq \bar y^0 \defeq y^0 \in \Dom \subdiff F^*$.
        \State Let $i \defeq 0$.
        \Repeat
            \State Let{\abovedisplayskip=0ex
                \begin{equation*}
                \left\{
                \begin{aligned}
                \tilde\tau_i & \defeq \tau_i/[1+\gamma\tau_i(\inv\lambda_i-1)],
                \\
                \tilde\sigma_{i+1} & \defeq \sigma_{i+1}/[1+\rho\sigma_{i+1}(\inv\mu_{i+1}-1)].
                \end{aligned}
                \right.
            \end{equation*}}
            \State Compute
                \begin{equation*}
                    \left\{
                    \begin{aligned}
                        \this{x}_\gamma & \defeq [\this{\bar x}+\gamma\tau_i(\inv\lambda_i-1)\thisx]/[1+\gamma\tau_i(\inv\lambda_i-1)],
                        \\
                        \nextx & \defeq \prox_{\tilde\tau_i G}(
                            \this{x}_\gamma
                            -\tilde\tau_i K^*\this{\tilde y}
                            ),
                        \\
                        \nexxt{\bar x} & \defeq \nextx+\lambda_{i+1}(\inv\lambda_i-1)(\nextx-\thisx),
                        \\
                        \nexxt{\tilde x} & \defeq \nextx + (\inv\lambda_i-1)(\nextx-\thisx),
                        \\
                        \this{y}_\rho & \defeq [\this{\bar y}+\rho\sigma_{i+1}(\inv\mu_{i+1}-1)\thisy]/[1+\rho\sigma_{i+1}(\inv\mu_{i+1}-1)],
                        \\
                        \nexty & \defeq \prox_{\tilde\sigma_{i+1} F^*}(
                            \this{y}_\rho
                            +\sigma_{i+1} K[\nexxt{\tilde x} + \omega_i(\nexxt{\tilde x}-\this{\tilde x})]
                            ),
                        \\
                        \nexxt{\bar y} & \defeq \nexty+\mu_{i+2}(\inv\mu_{i+1}-1)(\nexty-\thisy),
                        \\
                        \nexxt{\tilde y} & \defeq \nexty + (\inv\mu_{i+1}-1)(\nexty-\thisy).
                    \end{aligned}
                    \right.
                \end{equation*}%
            \State Update $i \defeq i +1$
        \Until a stopping criterion is satisfied.
    \end{algorithmic}

Observe: If $\gamma=0$, then $\tilde\tau_i=\tau_i$ and $\thisx_\gamma=\this{\bar x}$. If $\rho=0$, then $\tilde\sigma_{i+1}=\sigma_{i+1}$ and $\thisy_\rho=\this{\bar y}$.

\end{Algorithm}

\subsection{Basic conditions}

We now verify the basic conditions of \cref{thm:general-inertia} and the positivity of $\Test_{i+1}\Precond_{i+1}$.

\begin{lemma}
    \label{lemma:pdps-inertia-metric-update}
    With the setup \eqref{eq:pdps-inertia-structure}, the condition \eqref{eq:inertia-metric-cond} holds if
    \begin{subequations}%
    \label{eq:pd-lemma-cond-metric}
    \begin{align}%
        \label{eq:pd-lemma-cond-metric-exp-primal}
        \lambda_i^2\tauTest_i(1+2\gamma\tau_i\inv\lambda_i) & \ge \lambda_{i+1}^2\tauTest_{i+1},
        \\
        \label{eq:pd-lemma-cond-metric-exp-dual}
        \mu_{i+1}^2\sigmaTest_{i+1}(1+2\rho\sigma_{i+1}\inv\mu_{i+1}) & \ge \mu_{i+2}^2\sigmaTest_{i+2},
        \\
        \label{eq:final-metric-cond}
        \lambda_i\tauTest_i\tau_i
        & =
        \mu_{i}\sigmaTest_{i}\sigma_{i}.
    \end{align}%
    \end{subequations}%
\end{lemma}

\begin{proof}
    Inserting the operators from \eqref{eq:pdps-inertia-structure}, the condition \eqref{eq:inertia-metric-cond} reads
    {\small\[
        \begin{pmatrix}
            \tauTest_i\lambda_i(\lambda_i+2\gamma\tau_i) I & \tauTest_i\lambda_i\tau_i K^* \\
            -\sigmaTest_{i+1}\mu_{i+1}\sigma_{i+1}(2+\omega_i) K & \sigmaTest_{i+1}\mu_{i+1}(\mu_{i+1}+2\rho\sigma_{i+1}) I
        \end{pmatrix}
        \ge
        \begin{pmatrix}
            \tauTest_{i+1}\lambda_{i+1}^2 I & -\tauTest_{i+1}\lambda_{i+1}\tau_{i+1} K^* \\
            -\sigmaTest_{i+2}\mu_{i+2}\sigma_{i+2}\omega_{i+1} K & \sigmaTest_{i+2}\mu_{i+2}^2 I
        \end{pmatrix}.
    \]}
    Further inserting $\omega_i$ and $\omega_{i+1}$ from \eqref{eq:pdps-inertia-precond}, and using \eqref{eq:final-metric-cond}, the off-diagonal components cancel out.
    The diagonal components that are left are simply  \eqref{eq:pd-lemma-cond-metric-exp-primal} and \eqref{eq:pd-lemma-cond-metric-exp-dual}.
\end{proof}

\begin{lemma}
    \label{lemma:pdps-zm-bound}
    Let $i \in \N$. If \eqref{eq:pdps-inertia-structure} and \eqref{eq:pd-lemma-cond-metric} hold, then  $\Test_{i+1}\Precond_{i+1}$ is self-adjoint. If, moreover,
    \begin{equation}
        \label{eq:pdps-inertia-zm-bound-cond}
        (1-\kappa)\mu_{i+1}^2\sigmaTest_{i+1} \ge \tauTest_i\tau_i^2 \norm{K}^2
        \quad\text{for some}\quad
        \kappa \in [0, 1),
    \end{equation}
    then $\Test_{i+1}\Precond_{i+1}$ is positive definite, more precisely
    \begin{equation}
        \label{eq:pdps-inertia-zm-bound}
        \Test_{i+1}\Precond_{i+1}
        \ge
        \delta \Test_{i+1}
        \quad\text{for}\quad
        \delta \defeq 1- \sqrt{1-\kappa}.
    \end{equation}
\end{lemma}

\begin{proof}
    For now, take arbitrary $\delta \in [0, \kappa]$.
    From \eqref{eq:pdps-inertia-precond}, using Cauchy's inequality
    \[
        \Test_{i+1}\Precond_{i+1}
        =
        \begin{pmatrix} \tauTest_i I
        & -\inv\mu_{i+1}\tauTest_i\tau_i K^* \\
        -\inv\mu_{i+1}\tauTest_i\tau_i K
        & \sigmaTest_{i+1}I
        \end{pmatrix}
        \ge
        \begin{pmatrix} \delta \tauTest_i I
        & 0 \\
        0
        & \sigmaTest_{i+1}I - \inv{(1-\delta)}\mu_{i+1}^{-2}\tauTest_i\tau_i^2 KK^*
        \end{pmatrix}.
    \]
    Clearly then $\Test_{i+1}\Precond_{i+1}$ is self-adjoint.
    Using \eqref{eq:pdps-inertia-zm-bound-cond}, we have
    \[
        \sigmaTest_{i+1}I - \inv{(1-\delta)}\mu_{i+1}^{-2}\tauTest_i\tau_i^2 KK^*
        \ge
        \sigmaTest_{i+1}I - \inv{(1-\delta)}(1-\kappa)\sigmaTest_{i+1} I
        = (\kappa-\delta)(1-\delta)^{-1}\sigmaTest_{i+1} I.
    \]
    To make a specific choice of $\delta$, we equate $\delta=(\kappa-\delta)(1-\delta)^{-1}$. This gives the quadratic equation $2\delta-\delta^2-\kappa=0$ with the solution $\delta=1- \sqrt{1-\kappa}$. The rest is trivial.
\end{proof}

\subsection{Gap unrolling and alignment}

We now derive a basic convergence estimate using \cref{thm:general-inertia}.
This involves verifying \eqref{eq:inertia-h-cond} for some $\mathcal{V}_{i+1}(\realoptu)$, and estimating the sum of the latter from below to yield a useful gap estimate.
For the statement of the next lemma, we recall the definition of the strong-convexity adjusted functions $G_\gamma$ and $(F^*)_\rho$ from \cref{sec:inertia-sc}, and the corresponding gap functional defined in \eqref{eq:gap-sc}.

\begin{lemma}
    \label{lemma:pdps-gap-unrolling}
    Suppose \eqref{eq:pdps-inertia-structure} and \eqref{eq:pd-lemma-cond-metric} hold. Take
    \begin{subequations}%
    \label{eq:pdps-inertial-gap-conds}
    \begin{align}%
        \label{eq:pdps-inertial-cond-primal}
        (1-\lambda_{i+1})\tauTest_{i+1}\tau_{i+1} & \le \tauTest_i\tau_i,
        && \lambda_0 =1, \\
        \label{eq:pdps-inertial-cond-dual}
        (1-\mu_{i+1})\sigmaTest_{i+1}\sigma_{i+1} & \le \sigmaTest_i\sigma_i \le \sigmaTest_{i+1}\sigma_{i+1},
        && \mu_0 =1,
        \quad\text{and}\\
        \label{eq:pdps-gap-alignment-cond}
        \tauTest_{i}\tau_{i} & =\sigmaTest_i\sigma_i
        && (i \in \N).
    \end{align}%
    \end{subequations}%
    Then for any $\realoptu=(\realoptx, \realopty) \in \inv H(0)$, the iterates generated by \cref{alg:pdps-inertial-explicit} satisfy
    \begin{equation}
        \label{eq:pdps-general-estimate}
        \frac{1}{2}\norm{z^N-\realoptu}^2_{\Lambda_{N+1}^*\Test_{N+1}\Precond_{N+1}\Lambda_{N+1}}
        +
        \tauTest_{N-1}\tau_{N-1}\gap_{\gamma,\rho}(u^N; \realoptu)
        \le
        C_0(\realoptu)
        \quad
        (N \ge 1),
    \end{equation}
    where for any $w^0 \in \subdiff (F^*)_\rho(y^0)$ we set
    \begin{equation*}
        C_0(\realoptu) \defeq \frac{1}{2}\norm{z^0-\realoptu}^2_{\Lambda_{1}^*\Test_{1}\Precond_{1}\Lambda_{1}}
        +\sigmaTest_{0}\sigma_{0}\iprod{w^0-K\realoptx}{y^0-\realopty}.
    \end{equation*}
\end{lemma}


\begin{proof}
    Observe that \cref{alg:pdps-inertial-explicit} explicitly requires $y^0 \in \Dom \subdiff F^*$, so some $w^0 \in \subdiff (F^*)_\rho(y^0; \realopty)$ exists. The proximal steps moreover ensure $\nexty \in \Dom \subdiff F^*$ and $\nextx \in \Dom \subdiff G$ for all $i \in \N$.

    By the defining \eqref{eq:inertia-z}, the auxiliary sequence $\{\thisz \defeq (\this{\zeta}, \this{\eta})\}_{i \in \N} \subset X \times Y$ satisfies
    \begin{equation}
        \label{eq:inertia-z-pdps}%
        \lambda_i\nexxt{\zeta}=\nextx-(1-\lambda_i)\thisx
        \quad\text{and}\quad
        \mu_{i+1}\nexxt{\eta}=\nexty-(1-\mu_{i+1})\thisy
        \quad (i \in \N).
    \end{equation}
    Since $\mu_0=1$ and $\eta^0=y^0$, the latter also works for $i=-1$ and any, superfluous, $y^{-1} \in Y$.
    We observe that
    \[
        \tilde H_{i+1}(\nextu)-\Gamma_{i+1}(\nextu-\realoptu)
        =
        \begin{pmatrix}
            \tau_i[\subdiff G(\nextx) - \gamma(\nextx-\realoptx)
            +K^*\realopty]
            \\
            \sigma_{i+1}[\subdiff F^*(\nexty) - \rho(\nexty-\realopty) - K\realoptx]
        \end{pmatrix}.
    \]
    Let us define (recall \cref{sec:inertia-sc})
    \begin{subequations}
    \begin{align}
        \label{eq:pdps-bar-g}
        \bar G_\gamma(x; \realoptu) & \defeq G(x) - \frac{\gamma}{2}\norm{x-\realoptx}^2 + \iprod{K^*\realopty}{x-\realoptx}
        =G_\gamma(x; \realoptx) + \iprod{K^*\realopty}{x-\realoptx},
        \ \text{and}
        \\
        \label{eq:pdps-bar-f}
        (\bar F^*)_\rho(y; \realoptu) & \defeq F^*(y) - \frac{\rho}{2}\norm{y-\realopty}^2 - \iprod{K\realoptx}{y-\realopty}
        =(F^*)_\rho(y; \realopty) - \iprod{K\realoptx}{y-\realopty}.
    \end{align}
    \end{subequations}
    Then $\bar G_\gamma$ and $(\bar F^*)_\rho$ are convex with $\bar G_\gamma(x; \realoptu) \ge \bar G_\gamma(\realoptx; \realoptu)$, and $(\bar F^*)_\rho(y; \realoptu) \ge (\bar F^*)_\rho(\realopty; \realoptu)$ for all $x \in X$ and $y \in Y$.
    Moreover, \eqref{eq:inertia-h-cond} holds with
    \begin{equation*}
        \begin{split}
        \mathcal{V}_{i+1}(\realoptu)
        &
        \defeq
        \inf~
        \iprod{\tilde H_{i+1}(\nextu) -\Gamma_{i+1}(\nextu-\realoptu)}{\nextz-\realoptu}_{\Lambda^*_{i+1}\Test_{i+1}}
        \\
        &
        =
        \inf~\bigl[
        \tauTest_i\tau_i\lambda_i \iprod{\subdiff \bar G_\gamma(\nextx; \realoptu)}{\nexxt{\zeta}-\realoptx}
        +
        \sigmaTest_{i+1}\sigma_{i+1}\mu_{i+1} \iprod{\subdiff (\bar F^*)_\rho(\nexty;\realoptu)}{\nexxt{\eta}-\realopty}
        \bigr].
        \end{split}
    \end{equation*}
    For each $i \in \N$, let  $\nexxt{\bar q} \in \subdiff \bar G_\gamma(\nextx; \realoptu)$, and let $\this{w} \in Y$ be such that $\this{\bar w} = \this{w} - K\realoptx \in \subdiff (\bar F^*)_\rho(\thisy; \realoptu)$.
    Define
    \begin{equation}
        \label{eq:pdps-inertia-gap-part-1}
        s_N \defeq s_N^G + s_N^{F^*}
        \defeq
        \sum_{i=0}^{N-1}
        \tauTest_i\tau_i\lambda_i \iprod{\nexxt{\bar q}}{\nexxt{\zeta}-\realoptx}
        +
        \sum_{i=0}^{N-1}
        \sigmaTest_{i+1}\sigma_{i+1}\mu_{i+1} \iprod{\nexxt{\bar w}}{\nexxt{\eta}-\realopty}.
    \end{equation}
    Since we have assumed \eqref{eq:pdps-inertia-structure} and \eqref{eq:pd-lemma-cond-metric}, we may use \cref{lemma:pdps-inertia-metric-update,lemma:pdps-zm-bound} to verify \eqref{eq:inertia-metric-cond} and the self-adjointness of $\Test_{i+1}\Precond_{i+1}$.
    We may therefore use \cref{thm:general-inertia} to establish \eqref{eq:pdps-general-estimate} if we further show that
    \begin{equation}
        \label{eq:pdps-inertia-gap-part-goal}
        s_N \ge
        \tauTest_{N-1}\tau_{N-1}\gap_{\gamma,\rho}(u^N; \realoptu) - c_0
    \end{equation}
    for
    \[
        c_0 \defeq 
        \sigmaTest_{0}\sigma_{0}\mu_{0} \iprod{\bar w^0}{\eta^0-\realopty}
        =
        \sigmaTest_{0}\sigma_{0}\iprod{w^0-K\realoptx}{y^0-\realopty}.
    \]    
    Indeed, this establishes the right hand side as a lower bound on $\sum_{i=0}^{N-1} \mathcal{V}_{i+1}(\realoptu)$.

    The difficulty in working with $s_N$ is that unless $\gamma=\rho=0$, our algorithm will give $\tauTest_i\tau_i=\sigmaTest_i\sigma_i$, not $\tauTest_i\tau_i=\sigmaTest_{i+1}\sigma_{i+1}$. We therefore have to realign variables.
    Using the assumption $\lambda_0=1$, \eqref{eq:pdps-inertial-cond-primal}, and \eqref{eq:inertia-z-pdps}, by \cref{lemma:h-sum-function} and \eqref{eq:pdps-bar-g} we have
    \begin{equation}
        \label{eq:pdps-inertia-gap-part-2}
        \begin{split}
        s_N^G
        &
        \ge \tauTest_{N-1}\tau_{N-1}[\bar G_\gamma(x^N; \realoptu)-\bar G_\gamma(\realoptx; \realoptu)]
        \\
        &
        =
        \tauTest_{N-1}\tau_{N-1}[G_\gamma(x^N; \realoptx)-G_\gamma(\realoptx; \realoptx)]
        +\tauTest_{N-1}\tau_{N-1}\iprod{K^*\realopty}{x^N-\realoptx}.
        \end{split}
    \end{equation}
    Regarding $s_N^{F^*}$, by the second inequality of \eqref{eq:pdps-inertial-cond-dual} we have $\sigmaTest_N\sigma_N \ge \sigmaTest_{N-1}\sigma_{N-1}$, Moreover, $\realopty$ minimises $(\bar F^*)_\rho(\freevar; \realoptu)$.
    Using these two facts after an application analogous to  \eqref{eq:pdps-inertia-gap-part-2} of \cref{lemma:h-sum-function}, we get
    \begin{equation}
        \label{eq:pdps-inertia-gap-part-3}
        \begin{split}
        s_N^{F^*} 
        &
        =
        \sum_{i=0}^N
        \sigmaTest_{i}\sigma_{i}\mu_{i} \iprod{\this{\bar w}}{\this{\eta}-\realopty} - c_0
        \\
        &        
        \ge \sigmaTest_N\sigma_N[(\bar F^*)_\rho(y^N; \realoptu)-(\bar F^*)_\rho(\realopty; \realoptu)] - c_0
        \\
        &
        \ge \sigmaTest_{N-1}\sigma_{N-1}[(\bar F^*)_\rho(y^N; \realoptu)-(\bar F^*)_\rho(\realopty; \realoptu)] - c_0
        \\
        &
        =  \sigmaTest_{N-1}\sigma_{N-1}[(F^*)_\rho(y^N; \realopty)-(F^*)_\rho(\realopty; \realopty)] -  \sigmaTest_{N-1}\sigma_{N-1}\iprod{K\realoptx}{y^N-\realopty} - c_0.
        \end{split}
    \end{equation}
    Combining \eqref{eq:pdps-inertia-gap-part-1}, \eqref{eq:pdps-inertia-gap-part-2}, and \eqref{eq:pdps-inertia-gap-part-3}, thus
    \[
        \begin{split}
        s_N
        &
        \ge
        \tauTest_{N-1}\tau_{N-1}[G_\gamma(x^N; \realoptx)-G_\gamma(\realoptx; \realoptx)]
        +\sigmaTest_{N-1}\sigma_{N-1}((F^*)_\rho(y^N; \realopty)-(F^*)_\rho(\realopty;\realopty))
        \\ \MoveEqLeft[-1]
        +\tauTest_{N-1}\tau_{N-1}\iprod{K^*\realopty}{x^N-\realoptx}
        -\sigmaTest_{N-1}\sigma_{N-1}\iprod{K\realoptx}{y^N-\realopty}
        - c_0.
        \end{split}
    \]
    Now \eqref{eq:pdps-gap-alignment-cond} establishes \eqref{eq:pdps-inertia-gap-part-goal}.
\end{proof}

\subsection{Step length and inertial parameter rules}
\label{sec:pdps-step-lengths}

We now consider several cases of the factors of (strong) convexity $\rho$ and $\gamma$ being zero or positive.
Throughout, as in the proof of \eqref{lemma:pdps-gap-unrolling}, we write $\thisz=(\this{\zeta}, \this{\eta})$, for the auxiliary sequence $\{\thisz\}_{i \in \N}$ defined in \eqref{eq:inertia-z}.
We first summarise the various lemmas and their conditions from above.

\begin{lemma}
    \label{lemma:pdps-final-general}
    With $\lambda_0 = 1$ and $\tau_0,\sigma_0,\tauTest_0,\sigmaTest_0>0$, suppose that $\mu_i=\lambda_i$ as well as
    \begin{subequations}%
    \label{eq:pdps-final-general-conds}
    \begin{align}
        \label{eq:pdps-final-general-conds-1}
        \sigmaTest_i\sigma_i &= \tauTest_{i}\tau_{i},
        & \omega_i\lambda_{i+1}\tauTest_{i+1}\tau_{i+1} & = \lambda_i\tauTest_i\tau_i.
        \\
        \label{eq:pdps-final-general-conds-2}
        \lambda_i^2\tauTest_i(1+2\gamma\tau_i\inv\lambda_i) & \ge \lambda_{i+1}^2\tauTest_{i+1},
        &
        (1-\lambda_{i+1})\tauTest_{i+1}\tau_{i+1} & \le \tauTest_i\tau_i \le \tauTest_{i+1}\tau_{i+1},
        \\
        \label{eq:pdps-final-general-conds-3}
        \lambda_{i}^2\sigmaTest_{i}(1+2\rho\sigma_{i}\inv\lambda_{i}) & \ge \lambda_{i+1}^2\sigmaTest_{i+1},
        &
        (1-\kappa)\lambda_{i+1}^2\sigmaTest_{i+1} & \ge \tauTest_i\tau_i^2 \norm{K}^2
        \qquad(i \in \N)
    \end{align}%
    \end{subequations}%
    for some $\kappa \in [0, 1)$.
    Then the iterates generated by \cref{alg:pdps-inertial-explicit} and the auxiliary sequence generated by \eqref{eq:inertia-z} satisfy with $\delta \defeq 1-\sqrt{1-\kappa}$ for any $N \ge 1$ and any $\realoptu \in \inv H(0)$ the estimate
    \begin{equation}
        \label{eq:pdps-general-estimate-simplified}
        \frac{\delta\tauTest_N\lambda_N^2}{2}\norm{\zeta^N-\realoptx}^2
        +\frac{\delta\sigmaTest_{N+1}\lambda_{N+1}^2}{2}\norm{\eta^N-\realopty}^2
        +
        \tauTest_{N-1}\tau_{N-1}\gap_{\gamma,\rho}(u^N; \realoptu)
        \le
        C_0(\realoptu).
    \end{equation}
\end{lemma}

\begin{proof}
    We first show that the setup \eqref{eq:pdps-inertia-structure} and the conditions \eqref{eq:pd-lemma-cond-metric}, \eqref{eq:pdps-inertia-zm-bound-cond}, and \eqref{eq:pdps-inertial-gap-conds} hold.
    Indeed, the second part of \eqref{eq:pdps-final-general-conds-1} is simply the choice of $\omega_i$ in \eqref{eq:pdps-inertia-precond}, while the rest of \eqref{eq:pdps-inertia-precond} follows from the derivation of \cref{alg:pdps-inertial-explicit} from this structural setup in \cref{sec:pdps-derivation}.
    Moreover, since $\mu_i=\lambda_i$, the first part of \eqref{eq:pdps-final-general-conds-1} implies \eqref{eq:final-metric-cond} and \eqref{eq:pdps-gap-alignment-cond}.
    Likewise, \eqref{eq:pdps-final-general-conds-2} implies  \eqref{eq:pd-lemma-cond-metric-exp-primal} and \eqref{eq:pdps-inertial-cond-primal}.
    The conditions \eqref{eq:pdps-final-general-conds-3} in turn imply \eqref{eq:pd-lemma-cond-metric-exp-dual} and \eqref{eq:pdps-inertia-zm-bound-cond}. Together  \eqref{eq:pdps-final-general-conds-1} and \eqref{eq:pdps-final-general-conds-2} imply \eqref{eq:pdps-inertial-cond-dual}.
    Therefore \eqref{eq:pd-lemma-cond-metric}, \eqref{eq:pdps-inertia-zm-bound-cond}, and \eqref{eq:pdps-inertial-gap-conds} hold in their entirety.
    We can thus apply \cref{lemma:pdps-inertia-metric-update,lemma:pdps-gap-unrolling} to obtain the estimate \eqref{eq:pdps-general-estimate}.
    By application of \cref{lemma:pdps-zm-bound} we then derive \eqref{eq:pdps-general-estimate-simplified} from \eqref{eq:pdps-general-estimate}.
\end{proof}



\begin{theorem}
    \label{thm:step-length-rules-gamma=0-rho=0}
    Suppose $\gamma=0$ and $\rho=0$.
    Take $\tau_0, \sigma_0>0$ with $\tau_0\sigma_0\norm{K}^2 < 1$, $\lambda_0=\mu_0=1$, $\epsilon \in [0, 1)$, and update ($i \in \N$)
    \begin{subequations}%
    \label{eq:step-length-choices-gamma=0-rho=0}%
    \begin{align}
        \tau_{i+1} & \defeq \tau_i\inv\lambda_i\lambda_{i+1}, 
        &
        \omega_i & \defeq 1,
        \\
        \sigma_{i+1} & \defeq \sigma_i\inv\lambda_i\lambda_{i+1},
        &
        \mu_{i+1} & \defeq \lambda_{i+1} \defeq \lambda_i/(1+(1-\epsilon)\lambda_i).
    \end{align}%
    \end{subequations}%
    Then the iterates generated by \cref{alg:pdps-inertial-explicit} satisfy $\gap(u^N; \realoptu) \to 0$ at the rate $O(1/N)$ for any $\realoptu \in \inv H(0)$.
\end{theorem}

\begin{proof}
    We will use \cref{lemma:pdps-final-general}, for which we need to verify \eqref{eq:pdps-final-general-conds}.
    We use \cref{lemma:tau-lambda-gamma-satisfaction}\,\cref{item:tau-lambda-gamma-satisfaction-3} to verify \eqref{eq:pdps-final-general-conds-2} for $\tauTest_i=\tau_i^{-2}$, $\tau_{i+1}=\tau_i(1-\lambda_{i+1})/(1-\epsilon\lambda_i)$, and $\lambda_{i+1}$ as in \eqref{eq:tau-lambda-gamma-satisfaction-3-rules}. With $\gamma=0$, the latter agrees with the expression for $\lambda_{i+1}$ in \eqref{eq:step-length-choices-gamma=0-rho=0}.
    With $\tauTest_i$ and $\rho=0$ inserted, the rest of \eqref{eq:pdps-final-general-conds} reads%
    \begin{subequations}%
    \label{eq:pdps-final-general-conds-gamma=0-rho=0}%
    \begin{align}%
        \label{eq:pdps-final-general-conds-gamma=0-rho=0-1}
        \sigmaTest_i\sigma_i &= \inv\tau_{i},
        & \omega_i\lambda_{i+1}\inv\tau_{i+1} & = \lambda_i\inv\tau_i.
        \\
        \label{eq:pdps-final-general-conds-gamma=0-rho=0-3}
        \lambda_{i}^2\sigmaTest_{i} & \ge \lambda_{i+1}^2\sigmaTest_{i+1},
        &
        (1-\kappa)\lambda_{i+1}^2\sigmaTest_{i+1} & \ge \norm{K}^2
        \qquad(i \in \N).
    \end{align}%
    \end{subequations}%
    Clearly the first part of \eqref{eq:pdps-final-general-conds-gamma=0-rho=0-3} holds by taking $\sigmaTest_i=\lambda_i^{-2}\inv\tau_0\inv\sigma_0$ for all $i \in \N$. Let us assume
    \begin{equation}
        \label{eq:pdps-final-general-conds-gamma=-rho=0-2}
        \omega_i=(\inv\lambda_{i+1}-1)/(\inv\lambda_i-\epsilon),
        \quad
        \tau_{i+1}=\tau_i\inv\lambda_i\lambda_{i+1}\omega_i,
        \quad\text{and}\quad
        \sigma_{i+1} \defeq \sigma_i\inv\lambda_i\lambda_{i+1}/\omega_i.
    \end{equation}
    Inserting $\tau_{i+1}=\tau_i(1-\lambda_{i+1})/(1-\epsilon\lambda_i)$ from above and $\omega_i$ from \eqref{eq:pdps-final-general-conds-gamma=-rho=0-2}, the second part of \eqref{eq:pdps-final-general-conds-gamma=0-rho=0-1} holds.
    It therefore only remains to secure the first part of \eqref{eq:pdps-final-general-conds-gamma=0-rho=0-1} and the second part of \eqref{eq:pdps-final-general-conds-gamma=0-rho=0-3}. With $\sigmaTest_i$ inserted, this is to say
    \begin{equation*}
        \sigma_i\tau_i = \sigma_0\tau_0\lambda_i^2
        \quad\text{and}\quad
        (1-\kappa) \ge \tau_0\sigma_0 \norm{K}^2.
    \end{equation*}%
    The second condition is simply our initial condition on the step lengths.
    The first condition holds if $\sigma_i=\inv\tau_i\lambda_i^2\sigma_0\tau_0$.
    Using that $\tau_{i+1}=\tau_i\inv\lambda_i\lambda_{i+1}\omega_i$, this holds when $\sigma_{i+1}$ as in \eqref{eq:pdps-final-general-conds-gamma=-rho=0-2}.
    We have therefore proved \eqref{eq:pdps-final-general-conds} to hold when \eqref{eq:pdps-final-general-conds-gamma=-rho=0-2} does, $\tauTest_i=\tau_i^{-2}$, $\sigmaTest_i=\lambda_i^{-2}\inv\tau_0\inv\sigma_0$, and $\tau_0\sigma_0\norm{K}^2 < 1$.

    Take now as stated $\omega_i=1$, and observe that the update rule for $\lambda_{i+1}$ in \eqref{eq:step-length-choices-gamma=0-rho=0} gives the rule for $\omega_i$ in \eqref{eq:pdps-final-general-conds-gamma=-rho=0-2}. Moreover, the rules for $\tau_{i+1}$ and $\sigma_{i+1}$ in \eqref{eq:step-length-choices-gamma=0-rho=0} are consistent with  \eqref{eq:pdps-final-general-conds-gamma=-rho=0-2}. Therefore  \eqref{eq:pdps-final-general-conds-gamma=-rho=0-2}, consequently  \eqref{eq:pdps-final-general-conds}, holds under the conditions of the theorem and the choices of the testing parameters $\tauTest_i$ and $\sigmaTest_i$ in the previous paragraph.
    \Cref{lemma:pdps-final-general} thus yields \eqref{eq:pdps-general-estimate-simplified}.
    Since $\gap_{\gamma,\rho}(u^N; \realoptu) \ge 0$ when $\realoptu \in \inv H(0)$, the growth estimate of \cref{lemma:tau-lambda-gamma-satisfaction}\,\cref{item:tau-lambda-gamma-satisfaction-3} in the case $\gamma=0$ applied in \eqref{eq:pdps-general-estimate-simplified} establishes the claimed convergence rate.
 \end{proof}

Thus, without any strong convexity, inertia and correction improve the ergodic $O(1/N)$ convergence of the gap for the PDPS to non-ergodic convergence.

\begin{theorem}
    \label{thm:step-length-rules-gamma>0-rho=0}
    Suppose $\gamma>0$ and $\rho=0$.
    Take $\epsilon \in [0, 1)$, $\tau_0, \sigma_0>0$ with $\tau_0\sigma_0\norm{K}^2 < 1$, initialise $\lambda_0 \defeq \mu_0 \defeq 1$, and update ($i \in \N$)
    \begin{subequations}%
    \label{eq:pdps-final-general-conds-gamma>0-rho=0}%
    \begin{align}
        \tau_{i+1} & \defeq \tau_i\inv\lambda_i\lambda_{i+1}\omega_i,
        &
        \omega_i & \defeq (\inv\lambda_{i+1}-1)/(\inv\lambda_i-\epsilon),
        \\
        \sigma_{i+1} & \defeq \sigma_i\inv\lambda_i\lambda_{i+1}/\omega_i,
        &
        \mu_{i+1} & \defeq \lambda_{i+1} \defeq \frac{\sqrt{\lambda_i^2+2\gamma\lambda_i\tau_i}}{1-\epsilon\lambda_i+\sqrt{\lambda_i^2+2\gamma\lambda_i\tau_i}}.
    \end{align}
    \end{subequations}
    Then the iterates generated by \cref{alg:pdps-inertial-explicit} satisfy both $\gap_{\gamma,0}(u^N; \realoptu) \to 0$ and $\norm{\zeta^N-\realoptx}^2 \to 0$ at the rate $O(1/N^2)$ for any $\realoptu \in \inv H(0)$.
\end{theorem}

\begin{proof}
    Note that $\lambda_{i+1}$ given in \eqref{eq:tau-lambda-gamma-satisfaction-3-rules} agrees with that in \eqref{eq:pdps-final-general-conds-gamma>0-rho=0}. Also note that in the proof of \cref{thm:step-length-rules-gamma=0-rho=0}, we did not involve the choices \eqref{eq:step-length-choices-gamma=0-rho=0} until the final paragraph. 
    Therefore, we may follow the proof of \cref{thm:step-length-rules-gamma=0-rho=0} to see \eqref{eq:pdps-final-general-conds} to hold when \eqref{eq:pdps-final-general-conds-gamma=-rho=0-2} does, $\tauTest_i=\tau_i^{-2}$, $\sigmaTest_i=\lambda_i^{-2}\inv\tau_0\inv\sigma_0$, and $\tau_0\sigma_0\norm{K}^2 < 1$.

    Observe now that \eqref{eq:pdps-final-general-conds-gamma=-rho=0-2} is consistent with the updates of $\omega_i$, $\tau_{i+1}$, and $\sigma_{i+1}$ in \eqref{eq:pdps-final-general-conds-gamma>0-rho=0}.
    By taking the testing parameters $\tauTest_i$ and $\sigmaTest_i$ as above, we have therefore verified  \eqref{eq:pdps-final-general-conds}, 
    so \cref{lemma:pdps-final-general} yields \eqref{eq:pdps-general-estimate-simplified}.
    The growth estimate of \cref{lemma:tau-lambda-gamma-satisfaction}\,\cref{item:tau-lambda-gamma-satisfaction-3} in the case $\gamma>0$ applied in \eqref{eq:pdps-general-estimate-simplified} establishes the claimed convergence rates.
\end{proof}

\begin{theorem}
    \label{thm:step-length-rules-gamma=0-rho>0}
    Suppose $\gamma=0$ and $\rho>0$.
    Take $\tau_0>0$ and $\epsilon \in [0, 1/2]$ with $\tau_0\norm{K}^2 < 2\rho$, initialise $\lambda_0 \defeq \mu_0 \defeq 1$, and update ($i \in \N$)
    \begin{align*}
        \tau_{i} & \defeq \tau_0,
        &
        \omega_i & \defeq (\inv\lambda_{i+1}-1)/(\inv\lambda_i-\epsilon)=\lambda_{i+1}\inv\lambda_i,
        \\
        \sigma_{i+1} & \defeq \frac{\lambda_i^2}{2\rho}
        &
        \mu_{i+1} & \defeq \lambda_{i+1} \defeq \frac{2}{1+\sqrt{1+4(\lambda_i^{-2}-\epsilon\inv\lambda_i)}},
    \end{align*}
    Then the iterates generated by \cref{alg:pdps-inertial-explicit} satisfy both $\gap_{0,\rho}(u^N; \realoptu) \to 0$ and $\norm{\eta^N-\realopty}^2 \to 0$ at the rate $O(1/N^2)$ for any $\realoptu \in \inv H(0)$.
\end{theorem}

\begin{proof}
    We use \cref{lemma:tau-lambda-gamma-satisfaction}\,\cref{item:tau-lambda-gamma-satisfaction-1} to verify \eqref{eq:pdps-final-general-conds-2} for $\tauTest_i=\lambda_i^{-2}$ as well as $\tau_i=\tau_0$ and $\lambda_{i+1}$ as stated. Inserting the $\tauTest_i$ and $\tau_i$, the rest of \eqref{eq:pdps-final-general-conds} now reduces to
    \begin{align*}
        \sigmaTest_i\sigma_i & = \lambda_i^{-2}\tau_0,
        &
        \omega_i\inv\lambda_{i+1}&=\inv\lambda_i,
        \\
        \lambda_{i}^2\sigmaTest_{i}(1+2\rho\sigma_{i}\inv\lambda_{i}) & \ge \lambda_{i+1}^2\sigmaTest_{i+1},
        \quad\text{and}
        &
        (1-\kappa) \lambda_{i+1}^2 \sigmaTest_{i+1}& \ge \lambda_i^{-2} \tau_0^2 \norm{K}^2.
    \end{align*}
    The second condition is one version of our update rule for $\omega_i$. We still need to show that the two versions of the rule are equal.
    If we take $\sigmaTest_i \defeq 2\rho \tau_0\lambda_{i-1}^{-2}\lambda_i^{-2}$ and $\sigma_i$ as stated, introducing the new variable $\lambda_{-1}$, not used in the algorithm, the rest becomes
    \begin{equation*}
        \lambda_{i-1}^{-2}+\inv\lambda_{i} \ge \lambda_i^{-2}
        \quad\text{and}\quad
        (1-\kappa)2\rho \ge \tau_0 \norm{K}^2
        \quad (i \in \N).
    \end{equation*}
    The latter condition is satisfied by our initial step length assumption for some $\kappa \in (0, 1)$. Since $\lambda_0=1$, the first condition holds for $i=0$ for any $\lambda_{-1}>0$. By \cref{lemma:lambda-recurrence}, $\lambda_{i-1}^{-2}-\epsilon\inv\lambda_{i-1}=\lambda_{i}^{-2}-\inv\lambda_{i}$ for $i \in \N$.
    Therefore the first condition holds, and the two expressions for $\omega_i$ are equivalent.

    We have thus verified  \eqref{eq:pdps-final-general-conds}, so \cref{lemma:pdps-final-general} gives the estimate \eqref{eq:pdps-general-estimate-simplified}.
    The growth estimate of \cref{lemma:tau-lambda-gamma-satisfaction}\,\cref{item:tau-lambda-gamma-satisfaction-1} applied there establish the claimed gap convergence rate. The convergence rate of the dual auxiliary variable is determined by the rate of growth of $\lambda_{N+1}^2\sigmaTest_{N+1}= 2\rho\tau_0\lambda_N^{-2}$.
    By the same  \cref{lemma:tau-lambda-gamma-satisfaction}\,\cref{item:tau-lambda-gamma-satisfaction-1}, $\tauTest_N\tau_N=\lambda_N^{-2}\tauTest_0$ grows at the rate $\Theta(N^2)$, so we get the claimed $O(1/N^2)$ convergence.
\end{proof}

\begin{theorem}
    \label{thm:step-length-rules-gamma>0-rho>0}
    Suppose $\gamma>0$ and $\rho>0$.
    Take $\lambda \in (0, 1)$ and $\epsilon \in [0, 1)$ with $\norm{K}^2 < 4\gamma\rho(\inv\lambda-\epsilon)(\inv\lambda-1)$.
    Update ($i \in \N$)
    \begin{align*}
        \tau_{i} & \defeq \lambda^2/[2\gamma(1-\lambda)],
        &
        \omega_i & \defeq (\inv\lambda-1)/(\inv\lambda-\epsilon),
        \\
        \sigma_{i} & \defeq \lambda^2/[2\rho(1-\lambda)],
        &
        \mu_{i+1} & \defeq \lambda_{i+1} \defeq \lambda.
    \end{align*}
    Then the iterates generated by \cref{alg:pdps-inertial-explicit} satisfy both $\gap_{\gamma,\rho}(u^N; \realoptu) \to 0$ and $\norm{z^N-\realoptu}^2 \to 0$ at a linear rate for (the unique) $\realoptu \in \inv H(0)$.
\end{theorem}

\begin{proof}
    We use \cref{lemma:tau-lambda-gamma-satisfaction}\,\cref{item:tau-lambda-gamma-satisfaction-2} to verify \eqref{eq:pdps-final-general-conds-2} for $\tauTest_i=c^i$ for $c \defeq (1-\epsilon\lambda)/(1-\lambda) > 1$ as well as $\tau_i \equiv \tau_0$ and $\lambda_{i+1} \equiv \lambda$ as stated. The rest of \eqref{eq:pdps-final-general-conds} now reduces to
    \begin{align*}
        \sigmaTest_i\sigma_i & = \tauTest_i \tau_0,
        &
        \omega_i c & = 1,
        \\
        \sigmaTest_{i}(1+2\rho\sigma_{i}\inv\lambda) & \ge \sigmaTest_{i+1},
        \quad\text{and}
        &
        (1-\kappa) \lambda^2 \sigmaTest_{i+1}& \ge \tauTest_i \tau_0^2 \norm{K}^2.
    \end{align*}
    Clearly our choice of $\omega_i=1/c$ satisfies the second condition.
    Taking $\sigmaTest_i=\tauTest_i\rho/\gamma$ and, as stated, $\sigma_i=\gamma\tau_0/\rho=\lambda^2/[2\rho(1-\lambda)]$, the first condition is also satisfied, while the third condition becomes $\tauTest_{i+1}(1+2\gamma\tau_0\inv\lambda) \ge \tauTest_i$. As $\lambda_i \equiv \lambda$, this is the first part of \eqref{eq:pdps-final-general-conds-2}, which we have already verified.
    The last condition becomes $(1-\kappa)\lambda^2\rho\inv\gamma c \ge \tau_0^2 \norm{K}^2$, which with $c$ and $\tau_0$ expanded is $4(1-\kappa)\gamma\rho(1-\epsilon\lambda)(1-\lambda) \ge \lambda^2 \norm{K}^2$. This is secured by our assumed bound on $\norm{K}$.

    We have thus verified  \eqref{eq:pdps-final-general-conds}, so \cref{lemma:pdps-final-general} yields the estimate \eqref{eq:pdps-general-estimate-simplified}.
    The growth estimate of \cref{lemma:tau-lambda-gamma-satisfaction}\,\cref{item:tau-lambda-gamma-satisfaction-2} applied there establish the claimed gap and primal variable convergence rates. Since $\sigmaTest_i=\tauTest_i\rho/\gamma$ and $\mu_i=\lambda_i$, the dual variable converges at the same rate as the primal variable.
\end{proof}

\begin{remark}[Partial gaps]
    \label{rem:pdps-partial-gap}
    Convergence of the Lagrangian gap is weak compared to the true duality gap.
    Suppose the product set $B_x \times B_y \subset X \times Y$ is bounded and contains some $\realoptu \in \inv H(0)$.
    In the literature, \term{partial gaps} are considered,
    \[
        0 \le
        \gap(u; B_x, B_y)
        \defeq
        \sup_{\optx \in B_x} \inf_{\opty \in B_y}~ \gap(u; (\optx, \opty)).
    \]
    Ergodic partial gap estimates can be derived in the non-inertial setting, see \cite{chambolle2010first}, because the fact that $\realoptu \in \inv H(0)$ is never actually needed in the proofs; $\realoptu \in X \times Y$ can be any element. Indeed, even in our work, the main reason we need to assume $\realoptu$ to be a solution are the final phases of the unrolling \cref{lemma:h-sum-function,lemma:h-sum-function-smooth}.
    However, because of this, we cannot derive partial gap estimates.
\end{remark}

\begin{remark}[Forward step]
    If we want to solve $\min_{x \in X} G(x)+E(x)+F(Kx)$, where $E$ is convex with $\grad E$ $L$-Lipschitz, using \cref{lemma:h-sum-function-smooth}, it is possible to incorporate into \cref{alg:pdps-inertial-explicit} a forward step with respect to $E$: we change the update of $\nextx$ into
    \[
         \nextx \defeq \prox_{\tilde\tau_i G}(
                            \this{x}_\gamma
                            -\tilde\tau_i[K^*\this{\tilde y}+\grad E(\this{\bar x})]
                            ).
    \]
    In this case, also minding \cref{lemma:pdps-zm-bound}, we have to ensure that $\tau_i L \le 1-\sqrt{1-\kappa}$. It is not difficult to verify that the update rule of \cref{lemma:tau-lambda-gamma-satisfaction}\,\ref{item:tau-lambda-gamma-satisfaction-3} satisfies $\lambda_{i+1} \in (\epsilon, 1)\lambda_i$, hence $\tau_{i+1} \le \tau_i$.

    In the proofs of \cref{thm:step-length-rules-gamma=0-rho=0,thm:step-length-rules-gamma>0-rho=0} we can take $\kappa \in (0, 1)$ such that $\sigma_0\tau_0\norm{K}^2=1-\kappa$, so we are led to the condition $\sqrt{\sigma_0\tau_0}\norm{K} + \tau_0 L \le 1$, which also ensures, hence replaces, the original bound $\sigma_0\tau_0\norm{K}^2 < 1$.

    Likewise, in the proof of \cref{thm:step-length-rules-gamma=0-rho>0}, $\tau_0 \norm{K}^2 = 2\rho(1-\kappa)$ leads to $\sqrt{\tau_0/(2\rho)} \norm{K} + \tau_0 L \le 1$.

    In the proof of \cref{thm:step-length-rules-gamma>0-rho>0}, $4(1-\kappa)\gamma\rho(\inv\lambda-\epsilon)(\inv\lambda-1) = \lambda^2 \norm{K}^2$ similarly leads to the replacement initialisation bound $(4\gamma\rho(\inv\lambda-\epsilon)(\inv\lambda-1))^{-1/2}\norm{K} + \tau_0 L \le 1$.
\end{remark}

\section{Numerical experience}
\label{sec:numerical}

We study the performance of the proposed algorithm on three image processing and inverse problems: denoising, sparse Fourier inversion, and Positron Emission Tomography (PET), all with total variation regularisation. We also performed experiments on deblurring, where the results were comparable to denoising. Denoising is the most basic image processing task, while sparse Fourier inversion is used for magnetic resonance image reconstruction; see, e.g., \cite{benning2015preconditioned,knoll2010second}.
These two problems are of the form
\begin{equation}
    \label{eq:denoising}
    \min_{x \in \R^{n_1n_2}}~ \frac{1}{2}\norm{z-Tx}_2^2 + \beta \norm{Dx}_{2,1},
\end{equation}
where $n_1 \times n_2$ is the size of the unknown image $x$ in pixels, $z \in \R^m$ is the corrupted data, and $\beta>0$ a regularisation parameter. The matrix $D \in \R^{2n_1n_2 \times n_1n_2}$ is a discretisation of the gradient operator, and $\norm{g}_{2,1} \defeq \sum_{i=1}^{n_1n_2} \sqrt{g_{i,1}^2+g_{i,2}^2}$ for $g=(g_{\cdot,1},g_{\cdot,2}) \in \R^{2n_1n_2}$.
We take $D$ as forward-differences with Neumann boundary conditions.

The operator $T \in \R^{k \times n_1n_2}$ depends on the problem in question: for denoising, $T=I$ is the identity
and for sparse Fourier inversion it is the composition $T=S\mathcal{F}$ with a sub-sampling operator $S \in \R^{k \times n_1n_2}$ and the discrete Fourier transform $\mathcal{F}$.
For denoising
$k=n_1n_2$, while for sparse Fourier reconstruction, $k \ll n_1n_2$.

To implement variants of the PDPS, we note that \eqref{eq:denoising} can in all three cases be written in the saddle point form
\[
    \min_{x \in \R^{n_1n_2}} \max_{y \in \R^{2n_1n_2}}~
        \frac{1}{2}\norm{z- T x}_2^2
        + \iprod{Dx}{y} - \delta_{\beta B}(y),
\]
where $B=B_{\R^2}^{n_1n_2}$ for $B_{\R^2}$ the Euclidean unit ball in $\R^2$. Since $T$ is in both cases related to a unitary operator, we can easily compute the proximal map of $G(x) \defeq \frac{1}{2}\norm{z- T x}_2^2$.

The PET problem is slightly different. We take as $T$ a discrete Radon transform, each $[Tx]_j$ being the integral of the image $x$ over a line with angle parameter $\theta_j$ and displacement $r_j$. As the efficient and precise realisation of such an operator in general cases is outside the scope of the present work, in our simplified setting, we consider only the four angles $\theta_j \in \{\text{0\textdegree, 45\textdegree, 90\textdegree, 135\textdegree}\}$ and displacements $r_j$ such that $Tx$ consists of all row sums, all column sums, and all diagonal sums of $x$ rewritten as a $n_1 \times n_2$ matrix. We also change the first fidelity term in \eqref{eq:denoising} to model, instead of Gaussian noise, Poisson noise.
Finally, we need to force $x \ge 0$. That is, our problem is
\begin{equation*}
    \min_{x \in [0, \infty)^{n_1n_2}}~ \iprod{Tx}{\mathbb 1} - \iprod{b}{\log(Tx+c)} + \beta \norm{Dx}_{2,1},
\end{equation*}
where $\mathbb{1} \defeq (1,\ldots,1) \in \R^k$, $b \in (0, \infty)^k$ is the measured data, and $c \in (0, \infty)^k$ is a background intensity, assumed known. The logarithm is applied componentwise.

Computing the proximal step with respect to the fidelity term is challenging due to the structure of $T$. We therefore write also this term as a conjugate, observing that $g_j(z) \defeq z - b_j\log(z+c_j)$ has the conjugate $g_j^*(\phi_j) = -b_j + c_j(1-\phi_j) + b_j \log(b_j/(1-\phi_j))$.
Introducing the additional upper bound $x \le 1$, this leads to
\[
    \min_{x \in \R^{n_1n_2}} \max_{(\phi, y) \in \R^k \times \R^{2n_1n_2}}~ \delta_{[0,1]^{n_1n_2}}(x)
        + \iprod{(Tx, Dx)}{(\phi, y)} - \left(\delta_{\beta B}(y) + \sum_{j=1}^k g_j^*(\phi_j)\right),
\]
Without the additional upper bound, this problem arranged as the prototype problem \eqref{eq:saddle} would have $G=\delta_{[0,\infty)^{n_1n_2}}$, which has the conjugate $G^*=\delta_{(-\infty,0]^{n_1n_2}}$. Although our algorithms guarantee $\nextx \in [0,\infty)^{n_1n_2}$, the conjugate will cause the true (non-Lagrangian) duality gap
\begin{equation}
    \label{eq:truegap}
    \tilde\gap(x, y) \defeq G(x)+F(Kx)+G^*(-K^*y)+F^*(y) \ge \gap(x, y; \realoptx, \realopty)
\end{equation}
to be infinite in practise. However, we wish to report the true duality gap instead of the Lagrangian duality gap, as it does not depend on knowing a solution $(\realoptx, \realopty)$. This is why we have added the upper bound $x \le 1$. Any greater upper bound would also work, giving a slightly different duality gap.

\subsection{Data}
\label{sec:datasetup}

We performed the numerical experiments on the first two of our models on the parrot image (\#23) from the free Kodak image suite photo, depicted in \cref{fig:kodim23-original} together with the corrupted data and restored images for the test problems.
We also performed some experiments (see \cref{fig:multi-results}) on all 24 images of this image suite. However, the effect of the exact image on the ranking of the tested algorithms is generally small.
The size of all the images is $n_1 \times n_2 = 768 \times 512$. To study scalability, we also scaled it down to $n_1 \times n_2 = 192 \times 128$ pixels. Together with the dual variable, the problem dimensions are therefore $768 \cdot 512 \cdot 3 = 1179648 \simeq 10^6$ and $128 \cdot 128 \cdot 3 = 49152 \approx 4 \cdot 10^4$.

%

For the denoising problem we added Gaussian noise with standard deviation $51$ ($-13.9$dB) to the original test image.
To remove the noise, we first choose $\beta=0.2$ (low regularisation parameter), and then $\beta=1$ (high regularisation parameter).
Following \cite{tuomov-tgvlearn}, we scale this parameter by the factor $0.25$ for operations on the downscaled image.
We also added noise in the other test problems to avoid \term{inverse crimes} \cite{mueller2012linear}. For sparse Fourier inversion, we used the same level of noise as for denoising.
The
sparse Fourier inversion experiments are only performed on the original non-down-scaled image with the regularisation parameter
$\beta=0.1$ (sparse Fourier inversion).

For the PET problem, instead of photographs, we use the Shepp--Logan phantom in \cref{fig:phantom}. This is because the limited number of angles encoded in $T$ (reduction of data to mere $2.3\%$ for the phantom) would not give a recognisable reconstruction of a more complex image. Moreover, the phantom is more relevant to the problem in question. As the resolution, we take $n_1 \times n_2 = 256 \times 256$. To obtain the simulated measurement data $b$, we apply Poisson noise to the row, column and diagonal sums in $Tx$, and then add the background $c \defeq \mathbb{1}$.

\begin{figure}
    \centering
    \subcaptionbox{Original\label{fig:kodim23-original}}{%
        \includegraphics[width=0.3\textwidth]{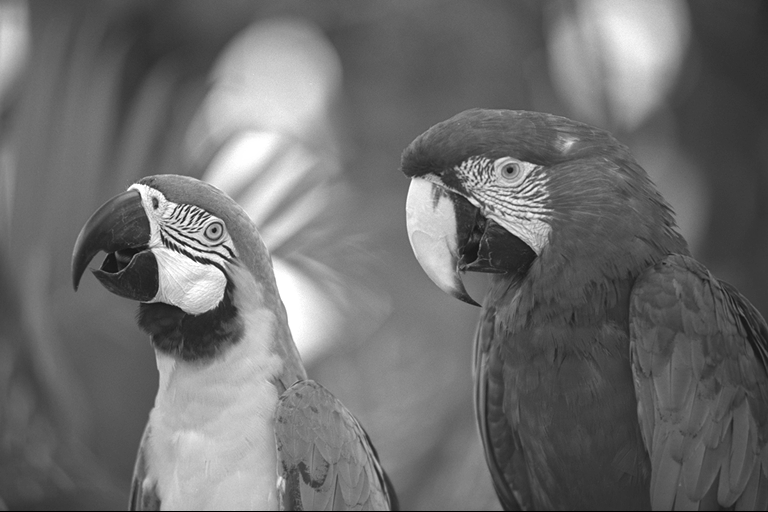}%
    }\,%
    \subcaptionbox{Noisy image}{%
        \includegraphics[width=0.3\textwidth]{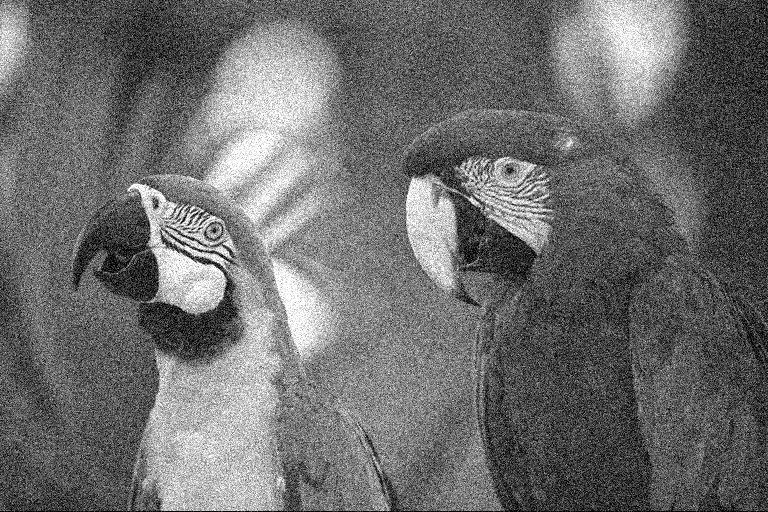}%
    }\,%
    \subcaptionbox{Denoised}{
        \includegraphics[width=0.3\textwidth]{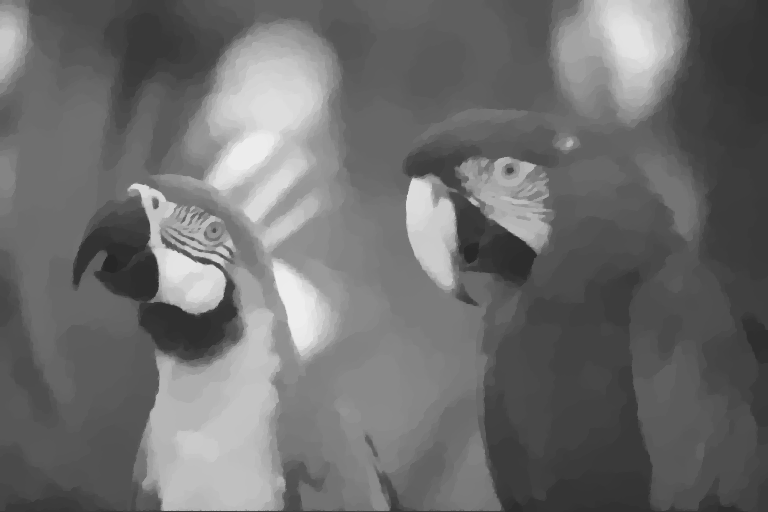}%
    }\\%
    \subcaptionbox{Fourier subsampling mask\label{fig:spiral}}{%
        \includegraphics[width=0.3\textwidth]{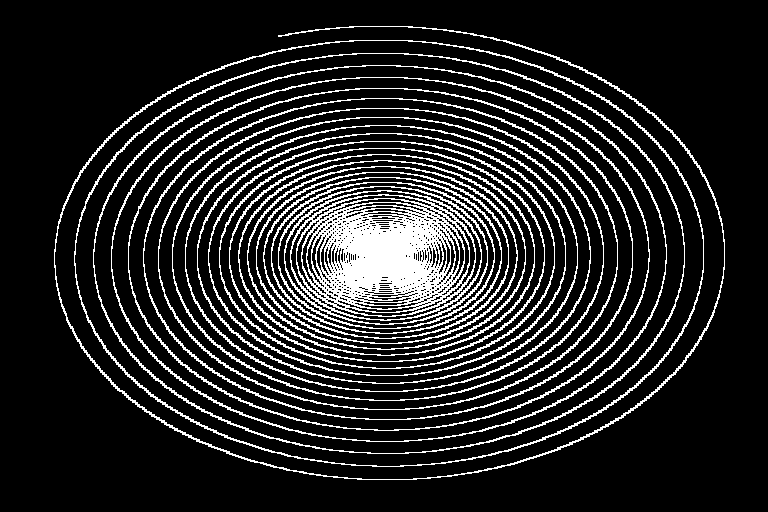}%
    }\,%
    \subcaptionbox{Zero-filling inversion\label{fig:kodim23-zerofill}}{%
        \includegraphics[width=0.3\textwidth]{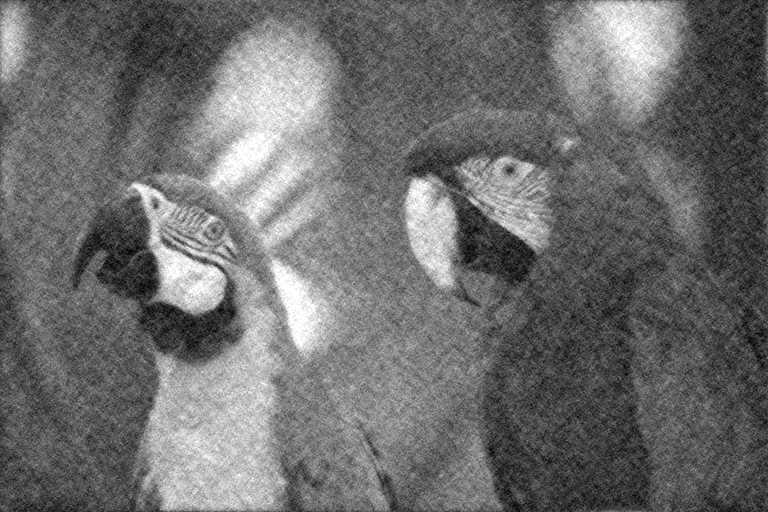}%
    }\,%
    \subcaptionbox{Sparse Fourier inversion}{%
        \includegraphics[width=0.3\textwidth]{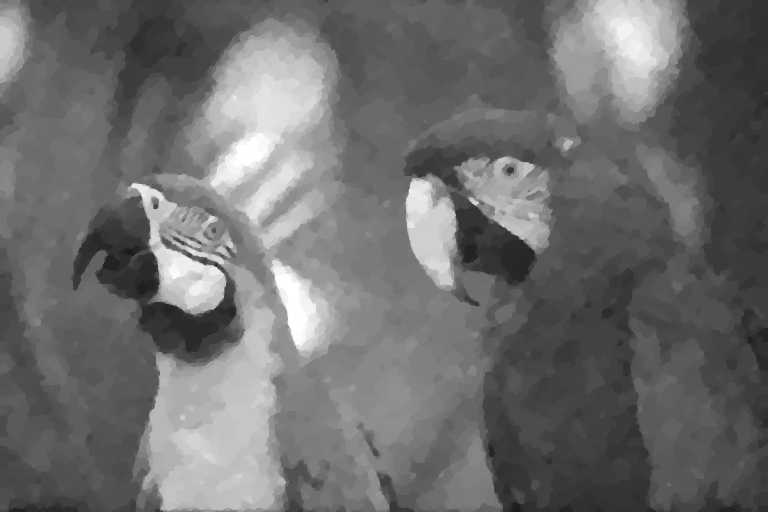}%
    }%
    \caption{Input data and reconstructions.
        The original image is \#23 from the free Kodak image suite, available online at the time of writing at \url{http://r0k.us/graphics/kodak/}.
        Since raw data $z$ for the sparse Fourier inversion is not visually informative, (\subref{fig:kodim23-zerofill}) displays the naïve zero-filling inversion $\mathcal{F}^*S^*z$ for the subsampling operator $S$ corresponding to the spiral mask in (\subref{fig:spiral}).}
    \label{fig:kodim23}
\end{figure}

\begin{figure}
    \centering
    \subcaptionbox{Phantom\label{fig:phantom-original}}{%
        \includegraphics[width=0.3\textwidth]{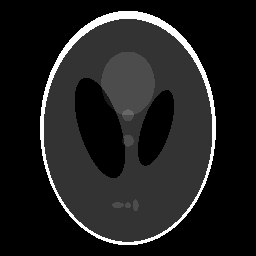}%
    }\,%
    \subcaptionbox{4-angle reconstruction\label{fig:phantom-reconstr}}{%
        \includegraphics[width=0.3\textwidth]{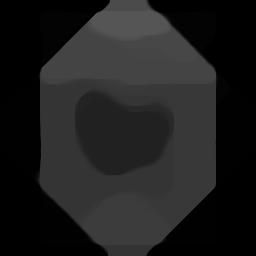}%
    }
    \caption{Shepp--Logan brain phantom and its reconstruction from simulated 4-angle (0\textdegree, 45\textdegree, 90\textdegree, 135\textdegree) positron emission tomography. The 4-angle tomography of a $256 \times 256$ image consists of $1534$ data points, meaning the reconstruction is achieved with just $2.3\%$ of data.}
    \label{fig:phantom}
\end{figure}

\subsection{Algorithmic setup}
\label{sec:numsetup}

We compare our algorithm (IC-PDPS) to the basic PDPS of \cite{chambolle2010first}, and the basic inertial (I-PDPS) and over-relaxed (R-PDPS) variants from \cite{chambolle2014ergodic}. The latter is essentially the V\~u--Condat algorithm. We do not include FISTA and other non-primal-dual algorithms in our comparisons, as of our example problems, they are easily applied only to TV denoising in its dual form. Similarly, the basic ADMM \cite{gabay} requires difficult inversions for our problems. Its more efficient preconditioned variant \cite{zhang2011unified}, on the other hand, is equivalent to the PDPS \cite{benning2015preconditioned}.

We use the same initial choices of $\tau_0=9.9/L$ and $\sigma_0=0.1/L$ with $L \defeq \sqrt{8} \ge \norm{D}$ \cite{chambolle2004algorithm} for all algorithms and the denoising and sparse Fourier inversion model problems. For the PET problem we take $\sigma_0=30/L'$ and $\tau_0=0.033/L'$ for an estimate $L' \ge \sqrt{\norm{T}^2+L^2}$. The ratio between $\tau_0$ and $\sigma_0$ has been hand-optimised for the baseline PDPS.
For the R-PDPS we take the additional over-relaxation parameter $\rho=1.5$.
For the I-PDPS we use fixed inertial parameter $\alpha=0.9/3$: according to \cite{chambolle2014ergodic}, the sequence of parameters $\{\alpha_i\}_{i \in \N}$ has to be non-decreasing with $\alpha_i<1/3$.
We also tested the FISTA rule, which did in practise yield better results for TV denoising, but completely failed for the other problems. Hence we use the provably convergent fixed parameter.

The denoising problem is strongly convex with factor $\gamma=1$, so we include results for both the unaccelerated and accelerated versions of the PDPS and IC-PDPS (\cref{thm:step-length-rules-gamma=0-rho=0,thm:step-length-rules-gamma>0-rho=0}). We also apply the rules of \cref{thm:step-length-rules-gamma=0-rho>0} to the problem with the primal and dual variables exchanged. This is denoted `dual IC-PDPS'.
The R-PDPS and the I-PDPS cannot with provable convergence be combined with strong convexity based acceleration: trying to do so was the starting point of our research. For acceleration we use $\gamma=0.5<1$, which is the maximal value for which the ergodic gap is known to converge at the rate $O(1/N^2)$ for the PDPS ($\gamma=1$ only yields convergence of the iterates; see \cite{chambolle2010first,tuomov-proxtest,tuomov-blockcp,tuomov-cpaccel}). For IC-PDPS $\gamma=1$ is allowed, and provably yields convergence of the gap, but in practise yields worse results than $\gamma=0.5$.

The IC-PDPS has one further parameter: $\epsilon \in [0, 1)$. For denoising and sparse Fourier inversion we generally take $\epsilon=0.7$, and for PET, $\epsilon=0.9$. We also report the denoising convergence behaviour for $\epsilon=0.5$ and $\epsilon=0$ in \cref{fig:tv-denoising-highreg-epsilon-comparison}.

For our reporting, we computed a target optimal solution $\realoptx$ by taking one million iterations of the basic PDPS.
However, the convergence of the basic PDPS for sparse Fourier inversion appears to be very slow: judging by the gap in \cref{fig:tv-sparsefft-gap-hq-zinit}, the IC-PDPS converges much faster, while both the PDPS and I-PDPS flatten out.
We therefore computed the target solution for sparse Fourier inversion by taking one million iterations of the IC-PDPS.
Note that the target solution is not used to compute the gap; instead of the Lagrangian duality gap \eqref{eq:gap}, we report true duality gap given in \eqref{eq:truegap}, as this does not depend on knowing a solution $(\realoptx,\realopty)$.

We report the distance to $\realoptx$ in decibels $10\log_{10}(\norm{x^i-\realoptx}^2/\norm{\realoptx}^2)$, as well as the duality gap, again in decibels relative to the initial gap as $10\log_{10}(\tilde\gap(\thisx, \thisy)^2/\tilde\gap(x^0, y^0)^2)$.
For the initial iterates we always took $x^0=0$ and $y^0=0$.
The hardware we used was a MacBook Pro with 16GB RAM and a 2.8 GHz Intel Core i5 CPU. The codes were written in MATLAB+C-MEX.

\subsection{Results}

The results for TV denoising of the downscaled image are in \cref{fig:tv-denoising-lowres}, and for the original image in \cref{fig:tv-denoising,tab:tv-denoising}. The latter includes both the high and low values of the regularisation parameter $\beta$. For the downscaled experiments we only report the lower value of $\beta$. The comparison for different values of $\epsilon$ for IC-PDPS is moreover in \cref{fig:tv-denoising-highreg-epsilon-comparison}, for the higher value of $\beta$.
The results for sparse Fourier inversion are in \cref{fig:tv-sparsefft,tab:tv-sparsefft}, and for PET in \cref{fig:tv-pet,tab:tv-kl}.
Finally, \cref{fig:multi-results} displays for denoising and sparse Fourier inversion the minimum and maximum interval for the duality gap over all 24 images in the image suite. We have excluded R-PDPS from these results to avoid overcrowding; its performance is comparable to I-PDPS, as can be gleaned from the other figures.

For TV denoising, the unaccelerated IC-PDPS is clearly the worst algorithm, while I-PDPS and R-PDPS slightly improve upon the basic PDPS. As expected from the $O(1/N)$ versus $O(1/N^2)$ convergence rates, all of these methods are significantly worse than the accelerated PDPS, the accelerated IC-PDPS, and the accelerated dual IC-PDPS. For the downscaled image and for low $\beta$ for the original resolution image, they are all comparable for the gap, but accelerated IC-PDPS somewhat surprisingly has asymptotically better iterate convergence. Of course, judging by the timings in \cref{tab:tv-denoising} in particular, the iterations of the IC-PDPS are somewhat more costly, so the basic accelerated PDPS appears the best choice in this case.

For high $\beta$, the results are initially similar, but both variants of the accelerated IC-PDPS are asymptotically better than the accelerated PDPS. This suggests that the IC-PDPS might perform better when there is ``more work to be done''.
This is somewhat confirmed by the results for sparse Fourier inversion, which is a significantly more difficult problem than TV denoising. There the gap convergence performance of IC-PDPS is significantly better than PDPS or I-PDPS: according to \cref{tab:tv-sparsefft}, compared to the PDPS only 75\% of the computational time is required to obtain $-35\text{dB}$ gap reduction.

For the PET problem, \cref{fig:tv-pet,tab:tv-kl} indicate that IC-PDPS has good gap convergence behaviour, taking 30\% less time than the PDPS to reach $-40\text{dB}$, but has primal variable convergence behaviour comparable to the PDPS. This indicates that the IC-PDPS has good convergence of the dual variable.

From \cref{fig:multi-results} we can see that the exact image does not significantly alter the rankings of the algorithms, with IC-PDPS performing significantly better than the other methods for sparse Fourier inversion.

\subsection{Conclusion}

While our proposed IC-PDPS does not always improve upon the basic, inertial, and over-relaxed PDPS, it never does significantly worse by iteration count. For some problems, such as sparse Fourier inversion and Positron Emission Tomography, it offers improved performance.
Moreover, we have theoretically guaranteed the $O(1/N)$ convergence of the Lagrangian gap functional or the $O(1/N^2)$ convergence of the strong convexity adjusted gap $\gap_{\gamma,\rho}$.
This is better than the merely ergodic convergence known of the PDPS and the basic inertial and over-relaxed variants.

\pgfplotsset{
    compat=1.3,
    every axis/.append style={
        ylabel shift=-8pt,
        xlabel shift=-2pt,
        ylabel near ticks,
    },
    every axis legend/.append style={
        inner xsep=1pt,
        inner ysep=1pt,
        nodes={inner ysep=1pt, inner xsep=1pt},
        legend image post style={scale=0.8}
    }
}

\def\tplot#1{\setlength{\figureheight}{0.27\textwidth}\setlength{\figurewidth}{0.42\textwidth}\scriptsize\hskip-6pt\input{#1}}


\begin{figure}[tbp!]
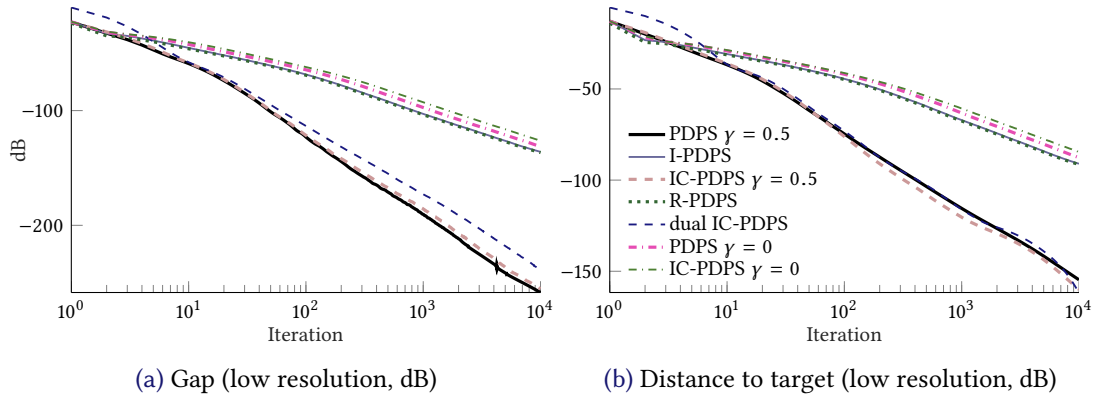

    \centering
    \subcaptionbox{Gap (low resolution, dB)\label{fig:tv-denoising-gap-lq-zinit}}{
        \setlength{\figureheight}{0.35\textwidth}\setlength{\figurewidth}{0.42\textwidth}\scriptsize\hskip-6pt\input{res/tv_denoising_results_lq_10000_gamma0.5_zinit,5_gaplog.tikz}%
    }%
    \subcaptionbox{Distance to target (low resolution, dB)\label{fig:tv-denoising-target-lq-zinit}}{
        \setlength{\figureheight}{0.35\textwidth}\setlength{\figurewidth}{0.42\textwidth}\scriptsize\hskip-6pt\input{res/tv_denoising_results_lq_10000_gamma0.5_zinit,5_target2log.tikz}%
    }
    \caption{Denoising convergence behaviour for low resolution image.}
    \label{fig:tv-denoising-lowres}
\end{figure}


\begin{figure}[tbp!]
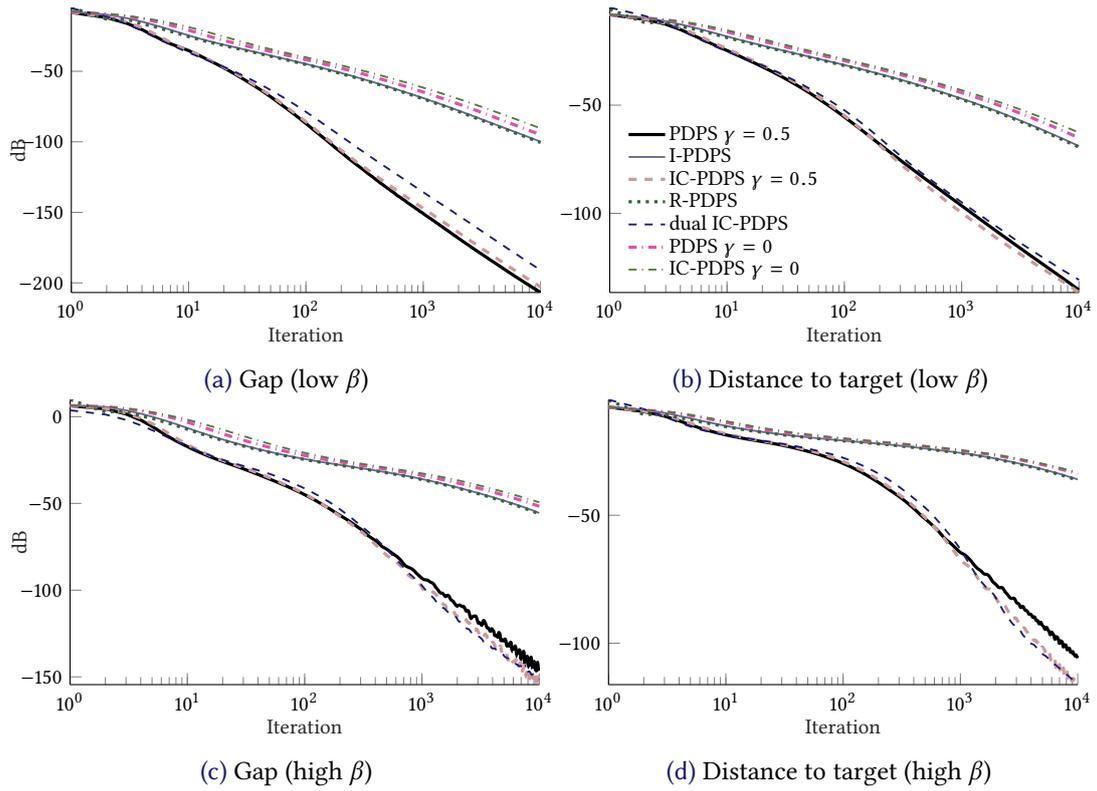

    \centering
    \subcaptionbox{Gap (low $\beta$)\label{fig:tv-denoising-gap-hq-zinit}}{
        \setlength{\figureheight}{0.35\textwidth}\setlength{\figurewidth}{0.42\textwidth}\scriptsize\hskip-6pt\input{res/tv_denoising_results_hq_10000_gamma0.5_zinit,5_gaplog.tikz}%
    }%
    \subcaptionbox{Distance to target (low $\beta$)\label{fig:tv-denoising-target-hq-zinit}}{
        \setlength{\figureheight}{0.35\textwidth}\setlength{\figurewidth}{0.42\textwidth}\scriptsize\hskip-6pt\input{res/tv_denoising_results_hq_10000_gamma0.5_zinit,5_target2log.tikz}%
    }%
    \\
    \subcaptionbox{Gap (high $\beta$)\label{fig:tv-denoising-gap-hq-zinit-highreg}}{
        \setlength{\figureheight}{0.35\textwidth}\setlength{\figurewidth}{0.42\textwidth}\scriptsize\hskip-6pt\input{res/tv_denoising_results_hq_10000_gamma0.5_zinit,5-highreg_gaplog.tikz}%
    }%
    \subcaptionbox{Distance to target (high $\beta$)\label{fig:tv-denoising-target-hq-zinit-highreg}}{
        \setlength{\figureheight}{0.35\textwidth}\setlength{\figurewidth}{0.42\textwidth}\scriptsize\hskip-6pt\input{res/tv_denoising_results_hq_10000_gamma0.5_zinit,5-highreg_target2log.tikz}%
    }%
    \caption{Denoising convergence behaviour.}
    \label{fig:tv-denoising}
\end{figure}

\begin{table}[tbp!]
    \caption{Denoising performance: CPU time and number of iterations (at a resolution of 10 after 100 iterations) to reach given duality gap and distance to target. The dashes indicate that the algorithm never reached (within the maximum number of iterations) the corresponding quality.}
    \label{tab:tv-denoising}

    \centering
    \small
    \setlength{\tabcolsep}{2pt}
    \captionsetup{position=top}

    \subcaptionbox{Low regularisation\label{tab:tv-denoising-lowreg}}{
        \begin{tabular}{l|rr|rr|rr|rr}
 & \multicolumn{2}{c}{gap $\le -40$dB} & \multicolumn{2}{c}{gap $\le -90$dB} & \multicolumn{2}{|c}{tgt $\le -40$dB} & \multicolumn{2}{|c}{tgt $\le -90$dB}\\
Method & iter & time & iter & time & iter & time & iter & time\\
\hline
PDPS $\gamma=0.5$ & 14 & 0.47s & 120 & 4.27s & 38 & 1.33s & 690 & 24.74s\\
I-PDPS & 58 & 2.39s & 4870 & 203.75s & 400 & 16.70s & -- & --\\
IC-PDPS $\gamma=0.5$ & 14 & 0.75s & 120 & 6.87s & 40 & 2.25s & 590 & 33.99s\\
R-PDPS & 55 & 2.36s & 4630 & 202.46s & 380 & 16.58s & -- & --\\
dual IC-PDPS & 13 & 0.64s & 160 & 8.53s & 43 & 2.25s & 750 & 40.18s\\
PDPS $\gamma=0$ & 82 & 2.87s & 6950 & 245.99s & 560 & 19.79s & -- & --\\
IC-PDPS $\gamma=0$ & 99 & 5.00s & 9710 & 494.99s & 650 & 33.09s & -- & --\\
\end{tabular}

    }

    \subcaptionbox{High regularisation\label{tab:tv-denoising-highreg}}{
        \begin{tabular}{l|rr|rr|rr|rr}
 & \multicolumn{2}{c}{gap $\le -40$dB} & \multicolumn{2}{c}{gap $\le -90$dB} & \multicolumn{2}{|c}{tgt $\le -40$dB} & \multicolumn{2}{|c}{tgt $\le -90$dB}\\
Method & iter & time & iter & time & iter & time & iter & time\\
\hline
PDPS $\gamma=0.5$ & 70 & 2.35s & 890 & 30.34s & 250 & 8.50s & 4330 & 147.73s\\
I-PDPS & 1810 & 70.74s & -- & -- & -- & -- & -- & --\\
IC-PDPS $\gamma=0.5$ & 73 & 3.85s & 740 & 39.51s & 270 & 14.38s & 2740 & 146.43s\\
R-PDPS & 1720 & 68.07s & -- & -- & -- & -- & -- & --\\
dual IC-PDPS & 91 & 4.78s & 770 & 40.85s & 330 & 17.48s & 2560 & 135.94s\\
PDPS $\gamma=0$ & 2580 & 84.57s & -- & -- & -- & -- & -- & --\\
IC-PDPS $\gamma=0$ & 3240 & 157.19s & -- & -- & -- & -- & -- & --\\
\end{tabular}

    }
\end{table}

\begin{figure}[tbp!]
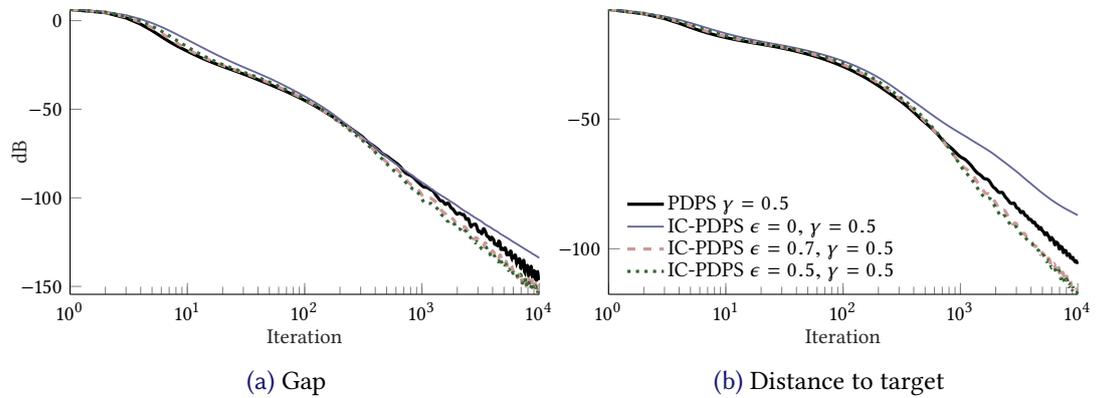

    \centering
    \subcaptionbox{Gap\label{fig:tv-denoising-gap-hq-zinit-highreg-epsilon-comparison}}{
        \setlength{\figureheight}{0.35\textwidth}\setlength{\figurewidth}{0.42\textwidth}\scriptsize\hskip-6pt\input{res/tv_denoising_results_hq_10000_gamma0.5_zinit,5-highreg_gaplog_epsilon_comparison.tikz}%
    }%
    \subcaptionbox{Distance to target\label{fig:tv-denoising-target-hq-zinit-highreg-epsilon-comparison}}{
        \setlength{\figureheight}{0.35\textwidth}\setlength{\figurewidth}{0.42\textwidth}\scriptsize\hskip-6pt\input{res/tv_denoising_results_hq_10000_gamma0.5_zinit,5-highreg_target2log_epsilon_comparison.tikz}%
    }%
    \caption{Effect of $\epsilon$ on denoising convergence behaviour.}
    \label{fig:tv-denoising-highreg-epsilon-comparison}
\end{figure}


\begin{figure}[tbp!]
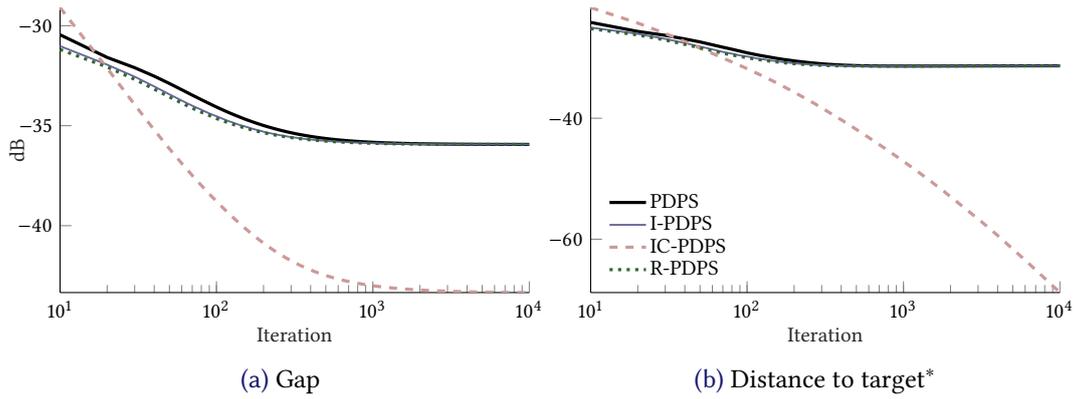

    \centering
    \subcaptionbox{Gap\label{fig:tv-sparsefft-gap-hq-zinit}}{
        \setlength{\figureheight}{0.35\textwidth}\setlength{\figurewidth}{0.42\textwidth}\scriptsize\hskip-6pt\input{res/tv_sparsefft_results_hq_10000_gamma0.5_zinit,5-ic_pdps_target_gaplog.tikz}%
    }%
    \subcaptionbox{Distance to target$^*$\label{fig:tv-sparsefft-target-hq-zinit}}{
        \setlength{\figureheight}{0.35\textwidth}\setlength{\figurewidth}{0.42\textwidth}\scriptsize\hskip-6pt\input{res/tv_sparsefft_results_hq_10000_gamma0.5_zinit,5-ic_pdps_target_target2log.tikz}%
    }%
    \caption{Sparse Fourier inversion convergence behaviour. $^*$Note: target computed by taking one million iterations of IC-PDPS instead of PDPS; see \cref{sec:numsetup}.}
    \label{fig:tv-sparsefft}
\end{figure}

\begin{table}[tbp!]
    \caption{
        Sparse Fourier inversion and PET performance: CPU time and number of iterations (at a resolution of 10) to reach given duality gap and distance to target. The dashes indicate that the algorithm never reached (within the maximum number of iterations) the corresponding quality.}
    \centering%
    \small%
    \captionsetup{position=top}%
    \setlength{\tabcolsep}{2pt}%
    \subcaptionbox{Sparse Fourier inversion\label{tab:tv-sparsefft}}{
        \begin{tabular}{l|rr|rr}
 & \multicolumn{2}{c}{gap $\le -35$dB} & \multicolumn{2}{|c}{tgt $\le -35$dB}\\
Method & iter & time & iter & time\\
\hline
PDPS & 210 & 15.96s & -- & --\\
I-PDPS & 150 & 11.87s & -- & --\\
IC-PDPS & 40 & 3.90s & 180 & 17.88s\\
R-PDPS & 140 & 12.12s & -- & --\\
\end{tabular}

    }%
    \subcaptionbox{PET\label{tab:tv-kl}}{
        \begin{tabular}{l|rr|rr}
 & \multicolumn{2}{c}{gap $\le -40$dB} & \multicolumn{2}{|c}{tgt $\le -30$dB}\\
Method & iter & time & iter & time\\
\hline
PDPS & 3740 & 43.31s & 6190 & 71.68s\\
I-PDPS & 3220 & 48.72s & 4380 & 66.28s\\
IC-PDPS & 2180 & 30.94s & 6010 & 85.32s\\
R-PDPS & -- & -- & 4150 & 57.32s\\
\end{tabular}

    }%
\end{table}

\begin{figure}[tbp!]
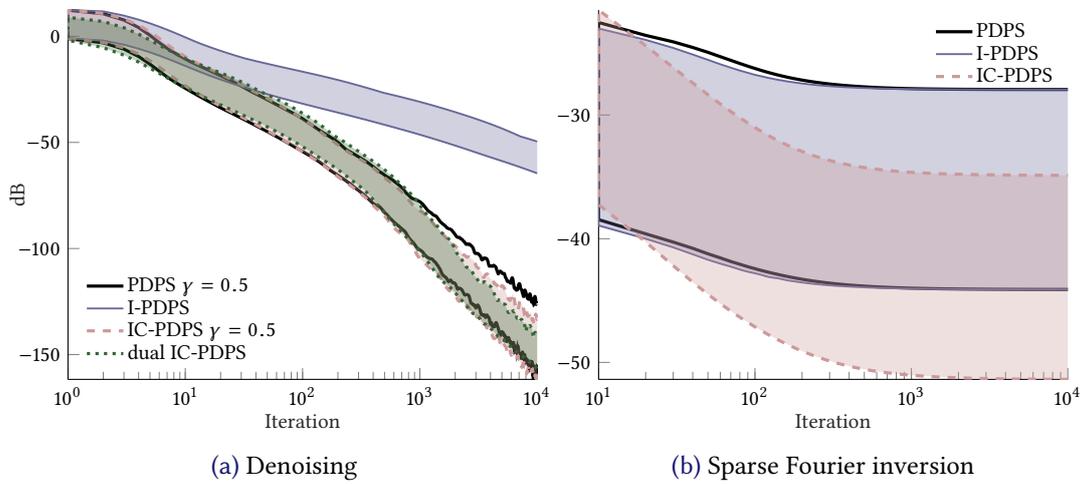

    \def\tplot#1{\setlength{\figureheight}{0.35\textwidth}\setlength{\figurewidth}{0.42\textwidth}\scriptsize\hskip-6pt\input{#1}}
    \centering
    \subcaptionbox{Denoising\label{fig:tv-denoising-hq-multi}}{
        \setlength{\figureheight}{0.35\textwidth}\setlength{\figurewidth}{0.42\textwidth}\scriptsize\hskip-6pt\input{res/tv_denoising_results_hq_10000_gamma0.5_zinit,5-multi-highreg_gaplog.tikz}%
    }%
    \subcaptionbox{Sparse Fourier inversion\label{fig:tv-sparsefft-hq-multi}}{
        \setlength{\figureheight}{0.35\textwidth}\setlength{\figurewidth}{0.42\textwidth}\scriptsize\hskip-6pt\input{res/tv_sparsefft_results_hq_10000_gamma0.5_zinit,5-multi_gaplog.tikz}%
    }%
    \caption{Gap convergence behaviour over multiple images (24). The filled areas indicate on each iteration the minimum and maximum gap (dB) over all the images.}
    \label{fig:multi-results}
\end{figure}


\begin{figure}[tbp!]
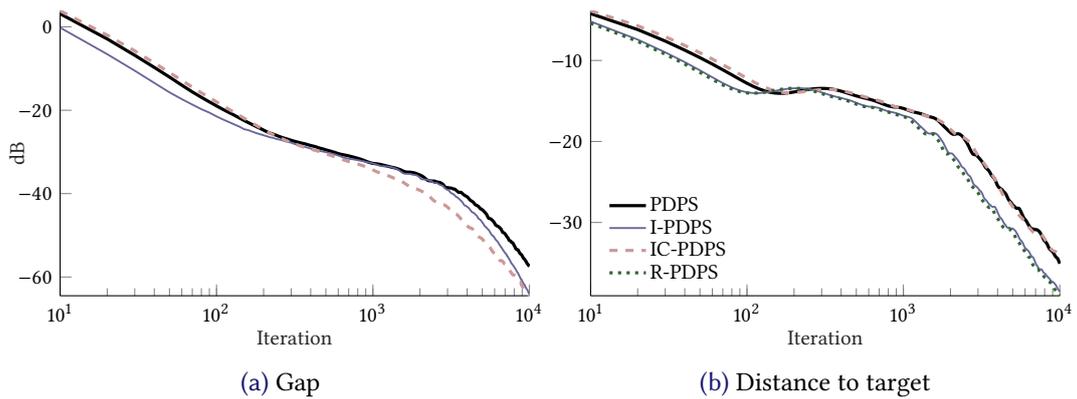

    \centering
    \subcaptionbox{Gap\label{fig:tv-kl-gap-hq-zinit}}{
        \setlength{\figureheight}{0.35\textwidth}\setlength{\figurewidth}{0.42\textwidth}\scriptsize\hskip-6pt\input{res/tv_kl_results_hq_10000_gamma0.5_zinit-revision_gaplog.tikz}%
    }%
    \subcaptionbox{Distance to target\label{fig:tv-kl-target-hq-zinit}}{
        \setlength{\figureheight}{0.35\textwidth}\setlength{\figurewidth}{0.42\textwidth}\scriptsize\hskip-6pt\input{res/tv_kl_results_hq_10000_gamma0.5_zinit-revision_target2log.tikz}%
    }%
    \caption{Convergence behaviour for the PET example problem.}
    \label{fig:tv-pet}
\end{figure}

\section*{Acknowledgements}

This research has been supported by the EPSRC First Grant EP/P021298/1, ``PARTIAL Analysis of Relations in Tasks of Inversion for Algorithmic Leverage'' as well as Academy of Finland grants 314701 and 320022.

\section*{\texorpdfstring{\normalsize}{}A data statement for the EPSRC}

Our algorithm implementations and the publicly available test images have been archived on Zenodo at \cite{tuomov-inertia-code}.

\def\bibnamefont#1{\textsc{\sffamily #1}}
\bibliographystyle{texbase/poor_approx_of_shinybib}

 \providecommand{\eprint}[1]{\href{http://arxiv.org/abs/#1}{arXiv:#1}}
  \providecommand{\eprint}[1]{\href{http://arxiv.org/abs/#1}{arXiv:#1}}
  \providecommand{\noopsort}[1]{}



\end{document}